\documentclass[11pt]{article}

\usepackage[english]{babel}
\usepackage[utf8x]{inputenc}
\usepackage[colorinlistoftodos]{todonotes}
\usepackage{float}

\usepackage{amsmath, amssymb, amsthm, cleveref}
\usepackage{graphicx}

\usepackage{mathrsfs}

\renewcommand{\baselinestretch}{1.25}

 \newtheorem{thm}{Theorem}[section]
 \newtheorem{cor}[thm]{Corollary}
 \newtheorem{lem}[thm]{Lemma}
 \newtheorem{prop}[thm]{Proposition}

 \newtheorem{obs}[thm]{Observation}
 \theoremstyle{definition}
 \newtheorem{defn}[thm]{Definition}
 \newtheorem{exmp}[thm]{Example}

 \theoremstyle{remark}
 \newtheorem{rem}[thm]{Remark}

\numberwithin{equation}{section}

 \newcommand{\To}{\longrightarrow}
 \newcommand{\Map}[3]{#1\, :\, #2\To #3}

 \newcommand{\Nat}{\mathbb{N}}
 
 \newcommand{\code}{\texttt{code}}
 \newcommand{\Sscl}{S_{\text{scl}}}
 \newcommand{\Tscl}{T_{\text{scl}}}
 \newcommand{\Spscl}{S'_{\text{scl}}}
  
  \newcommand{\Tpscl}{T'_{\text{scl}}}

 \newcommand{\set}[1]{\left\{#1\right\}}
 \newcommand{\Set}[2]{\set{#1\ \vert\ #2}}

\newcommand{\Mod}[3]{#1\equiv #2 \pmod{#3}}

\title{Oriented Bipartite Graphs and the Goldbach Graph}

\author{Sandip Das\thanks{\footnotesize{Advanced Computing and Microelectronic Unit, Indian Statistical Institute, Kolkata, India. \hspace{2in} sandipdas@isical.ac.in}}\,,
Prantar Ghosh\thanks{Department of Computer Science, Dartmouth College, Hanover, USA. prantar.ghosh.gr@dartmouth.edu }\,,
Shamik Ghosh\thanks{Department of Mathematics, Jadavpur University, Kolkata, India. shamik.ghosh@jadavpuruniversity.in\hspace{1in} (Communicating author)}\ \ and\ 
Sagnik Sen\thanks{Department of Mathematics, Indian Institute of Technology Dharwad, Dharwad, India. sen@iitdh.ac.in}}

\date{\today}

\begin{document}

\maketitle
\begin{abstract}
\noindent
In this paper, we study oriented bipartite graphs. In particular, we introduce ``bitransitive'' graphs. Several characterizations of bitransitive bitournaments are obtained. We show that bitransitive bitounaments are equivalent to acyclic bitournaments. As applications, we characterize acyclic bitournaments with Hamiltonian paths, determine number of non-isomorphic acyclic bitournaments of a given order, and solve the graph-isomorphism problem in linear time for acyclic bitournaments. Next, we prove the well-known Caccetta-H$\ddot{\textrm{a}}$ggkvist Conjecture for oriented bipartite graphs for some cases for which it is unsolved in general oriented graphs. We also introduce the concept of undirected as well as oriented ``odd-even'' graphs. We characterize bipartite graphs and acyclic oriented bipartite graphs in terms of them. In fact, we show that any bipartite graph (acyclic oriented bipartite graph) can be represented by some odd-even graph (oriented odd-even graph). We obtain some conditions for connectedness of odd-even graphs. This study of odd-even graphs and their connectedness is motivated by a special family of odd-even graphs which we call ``Goldbach graphs''. We show that the famous Goldbach's conjecture is equivalent to the connectedness of Goldbach graphs. Several other number theoretic conjectures (e.g., the twin prime conjecture) are related to various parameters of Goldbach graphs, motivating us to study the nature of vertex-degrees and independent sets of these graphs. Finally, we observe Hamiltonian properties of some odd-even graphs related to Goldbach graphs for small number of vertices.
\end{abstract}



\section{Introduction}
A (simple) directed graph $D=(V,E)$ is {\em bipartite} if the vertex set $V$ is partitioned into $X$ and $Y$ such that there is no arc between any two vertices of $X$ or between any two vertices of $Y$. We usually denote such a graph by $D=(X,Y,E)$. A directed bipartite graph $D=(V,E)$ is {\em oriented} if for any $u,v\in V$, $uv\in E$ implies $vu\notin E$. An oriented bipartite graph $D=(X,Y,E)$ is called a {\em bitournament} if for all $x\in X$ and $y\in Y$, either $xy\in E$ or $yx\in E$. For a directed graph $D$, the undirected graph $G(D)$ obtained from $D$ by disregarding directions of arcs is the {\em underlying graph} of $D$. Moreover two arcs $e$ and $f$ of $D$ are {\em adjacent} if they have a common end point in $G(D)$. The adjacency matrix $M(D)$ of a directed bipartite graph, $D=(X,Y,E)$ is of the following form:
$$M(D)\ =\ \begin{array}{cc|ccc|ccc|c}
\multicolumn{3}{c}{} & X & \multicolumn{2}{c}{} & Y & \multicolumn{2}{c}{}\\ \cline{3-8}
X &&& \mathbf{0} &&& A &&\\ \cline{3-8}
Y &&& B &&& \mathbf{0} &&\\ \cline{3-8}
\end{array}$$
where $A$ and $B$ are two $(0,1)$-matrices. Note that in the case of an undirected bipartite graph, we have $B=A^T$, but it is not true in general for oriented bipartite graphs. $D$ is {\em unidirectional} if either $xy\notin E$ for all $x\in X$, $y\in Y$ or $yx\notin E$ for all $x\in X$, $y\in Y$. In this case either $A=\mathbf{0}$ or $B=\mathbf{0}$ in $M(D)$.

\vspace{1em}\noindent
There are some interesting studies over directed bipartite graphs, oriented bipartite graphs, and in particular, oriented tress \cite{berr,min,mon,wang}. For a comprehesive study on bitournaments and in general, multipartite digraphs, one may consult the monograph \cite{BG}, in particular, Chapter 2 \cite{BFH} and Chapter 7 \cite{YEO}. In this paper, we introduce bitransitive (directed) graphs. Several characterizations of bitransitive bitournaments are obtained. In particular, we show that bitounaments are bitransitive if and only if they are acyclic. As applications of the theorem, we characterize acyclic bitournaments with Hamiltonian paths, determine number of non-isomorphic acyclic bitournaments of a given order, and solve the isomorphism problem in linear time for acyclic bitournaments. 

\vspace{1em}\noindent
Next, we consider the Caccetta-H$\ddot{\textrm{a}}$ggkvist Conjecture that states ``Every simple directed graph of order $n$ with minimum outdegree at least $r$ has a cycle of length at most $\lceil n/r\rceil$.'' The conjecture is open for $r=n/3,n/4,n/5$ and so on. We prove that the result is true for directed bipartite graphs for $r=n/3,n/4,n/5$. Our main technical contribution is the $n/5$ case, which requires some rigorous analysis of vertex-degrees. 

\vspace{1em}\noindent
We introduce the concept of oriented {\em odd-even} graphs and their undirected counterpart. We characterize the class of oriented bipartite graphs and (undirected) bipartite graphs in their terms. In fact, we show that any (acyclic oriented) bipartite graph can be represented by some (resp. oriented) odd-even graph. We obtain a necessary condition and another sufficient condition for connectedness of odd-even graphs. We study some cases where oriented odd-even graphs become unidirectional. 

\vspace{1em}\noindent
Finally, we introduce {\em Goldbach graphs}, a special family of odd-even graphs. We show that the famous Goldbach's conjecture is equivalent to the connectedness of Goldbach graphs. Furthermore, we observed that Maillet's, Kronecker's, and twin prime conjectures are related to various parameters of Goldbach graphs, especially to the vertex-degrees. So we study the nature of vertex-degrees and independent sets of Goldbach graphs. In the concluding section, we observe Hamiltonian properties of some odd-even graphs related to Goldbach graphs for small number of vertices and exhibit a sequence of even natural numbers up to $1000$ such that for any pair of consecutive numbers in the sequence, one of them is the sum of two odd primes or $1$ and the other is the difference between them (cf. Appendix B). 

\vspace{1em}\noindent
Throughout the paper let $\Nat$ denote the set of all natural numbers. If two natural numbers $a$ and $b$ are congruent modulo $p$, then we denote it by $a\equiv_p b$. Hence, in particular, for two numbers $a$ and $b$ with the same parity, we write $a \equiv_2 b$, and if they have the opposite parity, we write $a \not\equiv_2 b$. We denote the set $\set{1,2,\ldots,n}$ by $[n]$ for any $n\in\Nat$.

\section{Oriented Bipartite Graphs}

In this section, we study several classes of oriented bipartite graphs. In \Cref{sec:bitournbitran}, we introduce bitransitive  digraphs and characterize bitransitive bitournaments. Next, in \Cref{sec:acycbit}, we study acyclic bitournaments. Then, in \Cref{sec:ch}, we prove the Caccetta-H$\ddot{\textrm{a}}$ggkvist conjecture restricted to the class of oriented bipartite graphs for some cases that are open for general oriented graphs. 

\subsection{Bitournaments and Bitransitive Digraphs}\label{sec:bitournbitran}
We begin with an observation. Oriented trees form an interesting subclass of the class of oriented bipartite graphs. Let $T$ be an oriented tree. Then a path in the underlying tree $G(T)$ of $T$ is called {\em alternating} if each pair of adjacent arc are of opposite direction in $T$.

\begin{obs}
In an oriented tree $T$, there is an alternate path between any two vertices of $T$ if and only if for each vertex $v\in V(T)$, either $\text{indeg}\, (v) = 0$ or $\text{outdeg}\, (v) = 0$ (i.e., $T$ is unidirectional). 
\end{obs}

\noindent
In the following, we introduce {\em bitransitive} bipartite digraphs, analogous to transitive general digraphs.

\begin{defn}[Bitransitive Digraph]\label{defbit}
An oriented bipartite graph $D=(X,Y,E)$ is called {\em bitransitive} if for any $x_1,x_2,y_1,y_2\in X\cup Y$, $x_1y_1, y_1x_2, x_2y_2\in E \Longrightarrow x_1y_2\in E$ (see Figure \ref{bit1}). 
\end{defn}

\begin{figure}[ht]
\begin{center}
\includegraphics[scale=0.24]{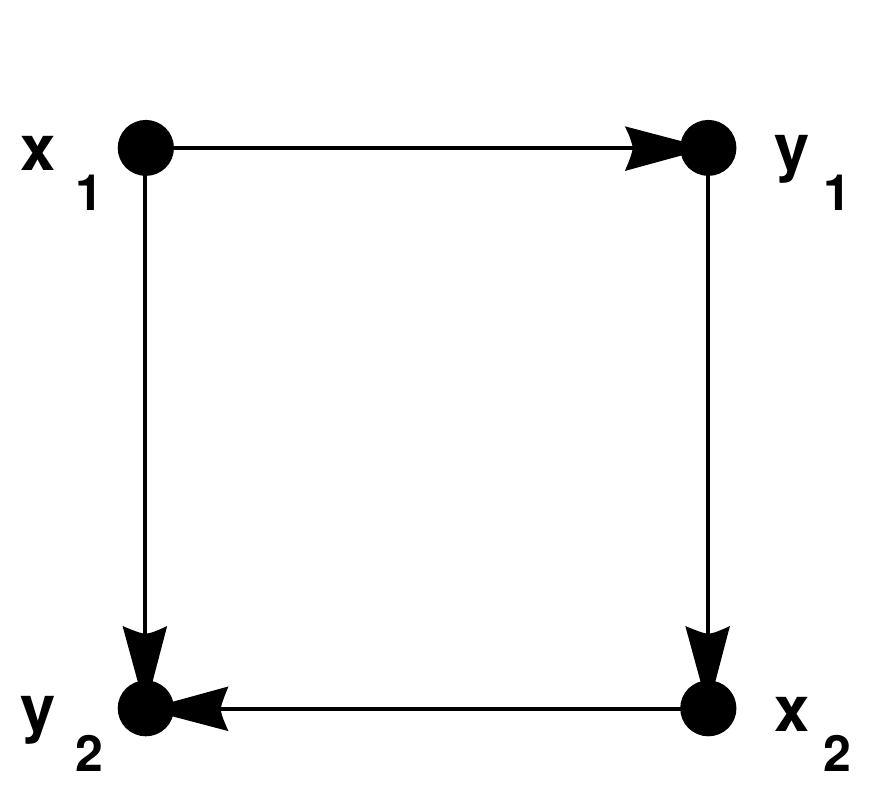}
\caption{An illustration of bitransitive property}\label{bit1}
\end{center}
\end{figure}

\noindent
We shall now define a digraph labelled by natural numbers that would work as an example of a bitransitive bitournament and shall help us to characterize them.
\begin{defn}[Digraph $D_S$]\label{DS}
Given a non-empty set $S\subseteq \Nat$, define $D_S$ as a digraph with the vertex set $S$ and the arc set $E=\Set{(a,b)\in S\times S}{a<b \text{ and } a \not\equiv_2 b}$. 
\end{defn}

\vspace{-1em}
\begin{exmp}
For any non-empty set $S\subseteq \Nat$, $D_S=(X,Y,E)$ is a bitransitive bitournament with $X=\Set{u\in S}{u \text{ is even}}$ and $Y=\Set{u\in S}{u \text{ is odd}}$.
\end{exmp}

\noindent
The following theorem characterizes bitransitive bitournaments. A {\em Ferrers digraph} $D=(V,E)$ is a directed graph whose successor sets are linearly ordered by inclusion where the successor set of $v\in V$ is its set of out-neighbors $\Set{u\in V}{vu\in E}$. It is known that a directed graph $D$ is a Ferrers digraph if and only if its adjacency matrix does not contain any $2\times 2$ permutation matrix (called a {\em couple})~\cite{bas,JR}:
$$\left[\begin{array}{cc}
1 & 0\\
0 & 1
\end{array}\right]\qquad\text{or}\qquad \left[\begin{array}{cc}
0 & 1\\
1 & 0
\end{array}\right].$$

\begin{thm}\label{bitbit}
Let $D = (X,Y,E)$ be a bitournament. Then the following are equivalent:
\begin{enumerate}
\item $D$ is bitransitive.
\item $D$ has no directed 4-cycle.
\item $D$ has no directed cycle.
\item The matrix $M(D)$ is given by $$\begin{array}{cc|ccc|ccc|c}
\multicolumn{3}{c}{} & X & \multicolumn{2}{c}{} & Y & \multicolumn{2}{c}{}\\ \cline{3-8}
X &&& \mathbf{0} &&& A &&\\ \cline{3-8}
\null &&& \null &&& \null &&\\[-1em]
Y &&& {\overline{A}}^T &&& \mathbf{0} &&\\ \cline{3-8}
\end{array}$$ where $A$ is the adjacency matrix of a Ferrer's digraph and $\overline{A}$ is the $1$'s complement of $A$.
\item $D\cong D_S$ (Definition \ref{DS})  for some nonempty set $S\subseteq\Nat$.
\end{enumerate} 
\end{thm}

\begin{proof}
$\mathbf{2 \implies 1}$: 
Suppose there is no directed $4$-cycle in a bitournament $D=(X,Y,E)$. Let $u_1u_2, u_2u_3, u_3u_4 \in E$ for some $u_1,u_2,u_3,u_4\in V(D)=X\cup Y$. Then $u_4u_1 \notin E$. Since $D$ is a bitournament, we have $u_1u_4 \in E$ . Hence it follows from Definition \ref{defbit} that $D$ is bitransitive.  

\vspace{1em}
\noindent
$\mathbf{1 \implies 3}$: 
Suppose $D=(X,Y,E)$ is bitransitive but has a directed cycle. Since $D$ is bipartite, there cannot be any odd cycle. Hence the cycle is even. Now let the cycle be $(u_1,u_2,\ldots,u_{2n})$. We prove by induction that $u_1u_{2k} \in E$ for all $k=1,2,\ldots,n$. By induction hypothesis, $u_1u_{2(k-1)} \in E$. Now $u_{2(k-1)}u_{2k-1}, u_{2k-1}u_{2k}$. Hence $u_1u_{2k} \in E$. So by induction, $u_1u_{2k} \in E$ for all $k=1,2,\ldots,n$. Hence $u_1u_{2n} \in E$. But we have already $u_{2n}u_1 \in E$. Since $D$ is a bitournament, both $u_1u_{2n}, u_{2n}u_1$ cannot be in $E$. Hence there is a contradiction. 

\vspace{1em}
\noindent
$\mathbf{3 \implies 2}$: Obvious.

\vspace{1em}
\noindent
$\mathbf{5 \implies 2}$:
Suppose $D \cong D_S$ for some nonempty set $S\subseteq\Nat$.
Suppose it has a directed $4$-cycle $(u_1,u_2,u_3,u_4)$. So  $u_1u_2, u_2u_3, u_3u_4, u_4u_1 \in E$. This implies $u_1 < u_2 < u_3 < u_4 < u_1$ which is a contradiction. So $D$ cannot have a directed $4$-cycle.

\vspace{1em}
\noindent
$\mathbf{2 \iff 4}$:
The adjacency matrix $A$ is not of a Ferrer's digraph if and only if there is a couple in $A$ such that 
$$\begin{array}{cc|ccc|ccc|c}
\multicolumn{3}{c}{} & y_r & \multicolumn{2}{c}{} & y_s & \multicolumn{2}{c}{}\\ \cline{3-8}
x_i &&& 1 &&& 0 &&\\ \cline{3-8}
x_j &&& 0 &&& 1 &&\\ \cline{3-8}
\end{array}$$
Hence ${\overline{A}}^T$ has the submatrix.
$$\begin{array}{cc|ccc|ccc|c}
\multicolumn{3}{c}{} & x_i & \multicolumn{2}{c}{} & x_j & \multicolumn{2}{c}{}\\ \cline{3-8}
y_r &&& 0 &&& 1 &&\\ \cline{3-8}
y_s &&& 1 &&& 0 &&\\ \cline{3-8}
\end{array}$$
Thus, $x_i \rightarrow y_r$, $y_r \rightarrow x_j$, $x_j \rightarrow y_s$ and $y_s \rightarrow x_i$. Then we get a $4$-cycle. Hence $A$ is not the adjacency matrix of a Ferrer's digraph if and only if there is a directed $4$-cycle. That is, $A$ is the adjacency matrix of a Ferrer's digraph if and only if there is no directed $4$-cycle. 

\vspace{1em}\noindent
$\mathbf{3 \implies 5}$:
We prove this by induction on number of vertices of a bitournament $D=(X,Y,E)$. The result is trivially true for $2$ vertices, one in each partite set. Now suppose there are $n+1>2$ vertices in $D$. Now we remove a vertex $v$ from $D$. Then by induction hypothesis, the result is true for the resultant graph, say $D_1$ which has $n$ vertices, i.e., $D_1\cong D_S$ for some $\emptyset\neq S\subseteq\Nat$. Now, let $A$ be the set vertices $u$ of $D$ such that there is a directed path from $u$ to $v$. Let $B$ be the set of vertices $w$ of $D$ such that there is a directed path from $v$ to $w$. Since there is no directed cycle, $A$ and $B$ are disjoint. Now in $D_S$, any two vertices of opposite parity are adjacent so they are belonging to different partite sets in $D$. Thus $v$ cannot be adjacent to both of them. Let $v\in X$. Without loss of generality we may assume that other vertices of $X$ are labeled by even numbers in $D_1$ for otherwise we increase the label of each vertex in $D_1$ by $1$. 

\vspace{1em}\noindent
Let $m$ be an even number that is greater than all labels of vertices in $D_1$. We label $v$ as $m$ and for each $w \in B$, we relabel $w$ as $w+m$. We first note that adding $m$ does not change the parity for any $w$ in $B$. Next we prove that this relabeling does not violate the adjacency condition. Let there be an arc from $w\in B$ to a vertex $x$ in $D_1$. Then by construction $x\in B$. Hence all arcs from any $w \in B$ go to vertices to $B$ itself. Since the original labeling did not violate the adjacency condition, increasing each label by $m$ also does not violate it for arcs from some vertex of $B$ to another vertex of $B$. Now for the arcs from some $x\notin B$ to some $w \in B$, the adjacency condition is not violated as we have increased the label of $w$. All arcs from $v$ go to some vertex of $B$. Since $v=m$ and $w+m>m$, the adjacency condition is not violated for arcs from $v$ to some vertex of $B$. If there is an arc from a vertex $x$ to $v$, then $x\in A$ and since the label of $v$ is higher than any vertex of $A$, the adjacency retains. In all other cases, labels are not changed. Hence the relabeling matches the adjacency condition of any arc in $D$. This completes the proof.
\end{proof}

\noindent
The above characterization of acyclic bitournaments in terms of digraphs $D_S$ enables us to characterize acyclic bitournaments with Hamiltonian paths, determine number of non-isomorphic acyclic bitournaments of a given order, and solve the graph isomorphism problem for acyclic bitournaments in linear time. We show this in the following section.

\subsection{Acyclic Bitournaments}\label{sec:acycbit}

In this section, we study the class of acyclic bitournaments (or, equivalently bitransitive bitournamemts). First, we show that an acyclic bitournament with a Hamiltonian path is unique (up to isomorphism) for a given order. Next, we show that the class of acyclic bitournaments can be given an ``encoding'' such that distinct (non-isomorphic) graphs from the class have distinct codes. This encoding enables us to count the number of non-isomorphic acyclic bitournaments and to check in linear time whether two given acyclic birtournaments are isomorphic.

\begin{thm}\label{bitham}
An acyclic bitournament $D$ with $n$ vertices has a Hamiltonian path if and only if $D$ is isomorphic to $D_{[n]}$, where $[n]:=\set{1,2,\ldots, n}$. 
\end{thm}

\begin{proof}
The ``if'' direction is immediate from the definition of $D_S$ (Definition \ref{DS}). In $D_{[n]}$, we have the Hamiltonian path $1\longrightarrow 2\longrightarrow \cdots \longrightarrow n$.

\vspace{1em}\noindent
For the ``only if'' direction, let $D$ be an acyclic bitournament which has a Hamiltonian path. By Theorem \ref{bitbit} (v), $D\cong D_S$ for some nonempty set $S\subseteq \Nat$. Hence, $D_S$ has a Hamiltonian path, say $a_1\longrightarrow a_2 \longrightarrow \ldots \longrightarrow a_n$. Now, in $D_S$, for every arc $xy$, we have $x<y$ and $x \not\equiv_2 y$. Thus, $a_i<a_{i+1}$ and $a_i \not\equiv_2 a_{i+1}$ for each $i=1,\ldots,n-1$. Therefore, all elements in the set $\{a_i \mid i \text{ is odd}\}$ have the same parity while all elements in the set $\{a_i \mid i \text{ is even}\}$ have the opposite parity, i.e., $i \not\equiv_2 j \iff a_i \not\equiv_2 a_j$. Let us map $i$ in $D_{[n]}$ to $a_i$ in $D_S$. We have 
\[i\rightarrow j \text{ in } D_{[n]}\iff i<j \text{ and } i \not\equiv_2 j \iff a_i < a_j \text{ and } a_i \not\equiv_2 a_j \iff a_i\rightarrow a_j \text{ in } D_S\] 
Hence, this is an isomorphism and $D_S \cong D_{[n]}$, i.e., $D\cong D_{[n]}$.
\end{proof}

\noindent
We now define a function that we shall use to encode acyclic bitournaments.
\begin{defn}[Function $\beta_S$]\label{def:scalingfunc}
Given a nonempty set $S\subseteq \Nat$, define the ``scaling'' function
$\beta_S$ as follows. Let the increasing order of the natural
numbers in $S$ be given by
$\langle a_1,a_2,\ldots, a_n \rangle$. Then, $\Map{\beta_S}{S}{\Nat}$ is defined inductively as $\beta_S(a_1)=1$, and for $i\geq 2$, $\beta_S(a_i) = \beta_S(a_{i-1})+1$ if $a_i \not\equiv_2 a_{i-1}$ and $\beta_S(a_i) = \beta_S(a_{i-1})+2$ if $a_i \equiv_2 a_{i-1}$.
\end{defn} 

\noindent
For a nonempty set $S\subseteq \Nat$ with $\langle a_1,a_2,\ldots, a_n \rangle$ being the increasing order of its elements, we define its ``scaled'' set $\Sscl$ as \[\Sscl:=\{ \beta_S(a_1), \beta_S(a_2),\ldots, \beta_S(a_n)\}\]

\begin{obs}\label{obs:scaleiso}
$D_S$ is isomorphic to $D_{\Sscl}$.
\end{obs}

\begin{proof}
For a digraph $D_S$, let the increasing order of the elements in $S$ be given by $\langle a_1,a_2,\ldots, a_n \rangle$ and that in $\Sscl$ be given by $\langle b_1, b_2,\ldots, b_n \rangle$. Then, by definition of $\Sscl$, for any $i,j\in [n]$, we have $b_i < b_j$ if and only if $a_i < a_j$ and $b_i \not\equiv_2 b_j$ if and only if $a_i \not\equiv_2 a_j$. Hence, by Definition \ref{DS}, $D_S \cong D_{\Sscl}$.
\end{proof}

\noindent
The following lemma enables us to give a unique code to each acyclic bitournament.

\begin{lem}\label{lem:codeuniq}
For two nonempty sets $S,T\subseteq \Nat$, $D_S$ is isomorphic to $D_T$ if and only if $\Sscl=\Tscl$. 
\end{lem}

\begin{proof}
One direction is obvious. $\Sscl=\Tscl\Longrightarrow D_{\Sscl}\cong D_{\Tscl}\Longrightarrow D_S\cong D_T$ (by Observation \ref{obs:scaleiso}).

\vspace{1em}\noindent
For the other direction, we proceed as follows. For a nonempty set $S$, denote the increasing order of its elements by $\langle a_1,\ldots,a_n\rangle$. Let $k$ be the integer such that $a_1,\ldots,a_k$ all have the same parity but $a_{k+1}$ has the opposite parity. Let $S':=S\setminus\{a_1\ldots,a_k\}$. We claim that $\Sscl=\{1,\ldots,2k-1\}\cup\{s'+2k-1: s'\in \Spscl\}$. 

\vspace{1em}\noindent
First note that by definition, $\beta_S(a_i)=2i-1$ for each $i\in [k]$. Let the increasing order of elements in $\Spscl$ be $\langle s'_1,\ldots,s'_{n-k}\rangle$. Now, we prove by induction that $\beta_S(a_{k+i})=2k-1+s'_i$ for each $i\in [n-k]$. Since $a_{k+1}\not\equiv_2 a_k$, we have $\beta_S(a_{k+1})=\beta_S(a_k)+1=2k-1+s'_1$. Therefore, the base case holds. Now suppose
$\beta_S(a_{k+j})=2k-1+s'_j$ for some $j\geq 1$. Then, if $a_{k+j}\not\equiv_2 a_{k+j+1}$, we have \[\beta_S(a_{k+j+1})=\beta_S(a_{k+j})+1=2k-1+s'_j+1=2k-1+\beta_{S'}(a_{k+j})+1=2k-1+\beta_{S'}(a_{k+j+1})=2k-1+s'_{j+1}\,.\]
Again, if $a_{k+j}\equiv_2 a_{k+j+1}$, we have \[\beta_S(a_{k+j+1})=\beta_S(a_{k+j})+2=2k-1+s'_j+2=2k-1+\beta_{S'}(a_{k+j})+2=2k-1+\beta_{S'}(a_{k+j+1})=2k-1+s'_{j+1}\,.\]
Thus, we have proven the claim by induction. Next, suppose $D$ is isomorphic to $D_S$ and $D_T$ for two sets $S$ and $T$. Thus, $D_S\cong D_T$. Hence, $|S|$ must equal $|T|$. We prove that $\Sscl=\Tscl$ by induction on $n=|S|=|T|$. The base case holds for $n=2$, since there is only one nonempty bitournament on two vertices which is a single arc, and $\Sscl=\Tscl=\{ 1,2\}$. Now suppose the result is true for all $m\leq n-1$. 

\vspace{1em}\noindent
Since $D_S$ and $D_T$ are acyclic, they must have nonzero {\em source} vertices, i.e., vertices with in-degree $0$. Again, since they are isomorphic, they must have the same number (say $k$) of source vertices. Note that since these are source vertices, they must have the least values in $S$ and $T$ by Definition \ref{DS} and all of them must have the same parity. Let the digraphs obtained by deleting the source vertices from each of $D_S$ and $D_T$ be isomorphic to $D_{S'}$ and $D_{T'}$ respectively. Then since $D_S\cong D_T$, we have $D_{S'}\cong D_{T'}$. Therefore, since $|S'|=|T'|=n-k$, by the induction hypothesis, we have $\Spscl=\Tpscl$. By the claim above, we have 
\[\Sscl = \{1,\ldots,2k-1\}\cup\{s'+2k-1: s'\in \Spscl\} = \{1,\ldots,2k-1\}\cup\{t'+2k-1: t'\in \Tpscl\} = \Tscl\, .\]
This completes the proof.
\end{proof}


\noindent
We are now ready to define the $\code$ of an acyclic bitournament.
\begin{defn}[$\code(D)$]\label{codedef}
Given an acyclic bitournament $D$, define $\code(D)$ as the sequence obtained by taking the elements of $\Sscl$ in increasing order, where $S$ is a set such that $D \cong D_S$.   
\end{defn}

\noindent
Note that by Lemma \ref{lem:codeuniq}, $\code(D)$ is a well-defined function. It then follows that $\code$ of an acyclic bitournament is unique up to isomorphism. 

\begin{lem}\label{lem:codeiso}
Two acyclic bitournaments $D_1$ and $D_2$ are isomorphic if and only if $\code(D_1)=\code(D_2)$.
\end{lem}

\begin{proof}
Let $D_1$ and $D_2$ be isomorphic to $D_S$ and $D_T$ for some nonempty sets $S, T \subseteq \Nat$ respectively. Then \[ D_1\cong D_2 \Longleftrightarrow D_S\cong D_T\Longleftrightarrow \Sscl =\Tscl\ (\text{\footnotesize{by Lemma \ref{lem:codeuniq}}})\Longleftrightarrow \code(D_1)=\code(D_2)\ (\text{\footnotesize{by Definition \ref{codedef}}}).\]\end{proof}
%
%

\noindent
Now, we shall use this encoding of the class of acyclic bitournaments to count of the number of non-isomorphic acyclic bitournaments of a given order.
\begin{thm}\label{bitiso}
Let $\alpha$ be the number of non-isomorphic acyclic bitournaments $D=(X,Y,E)$. Then, $\alpha = \binom{2n-1}{n}$ when $|X|=|Y|=n$ and $\alpha = \binom{m+n}{n}$ when $|X|=m\neq n=|Y|$.
\end{thm}

\begin{proof} Consider the case that the acyclic bitournament is $D=(X,Y,E)$ with $|X|=|Y|=n$. By Lemma \ref{lem:codeiso}, we see that the number of non-isomorphic acyclic bitournaments is the number of distinct \code s. Note that $\code(D)=\langle a_1,\ldots,a_{2n}\rangle$ has $a_1=1$ and must contain $n$ even numbers and $n$ odd numbers as $|X|=|Y|=n$. For $i\in [n-1]$ let $k_i$ denote the number of even integers between the $i$th and the $i+1$th odd numbers in the sequence $\langle a_1,\ldots,a_{2n}\rangle$, and $k_n$ be the number of even numbers after the $n$th odd number. We claim that two $\code$s $\langle a_1,\ldots,a_{2n}\rangle$ and $\langle a'_1,\ldots,a'_{2n}\rangle$ differ if and only if their corresponding sequences $\langle k_1, \ldots, k_n\rangle$ and $\langle k'_1, \ldots, k'_n\rangle$ differ.

\vspace{1em}\noindent
Let $a=\langle a_1,\ldots,a_{2n}\rangle$ and $a' = \langle a'_1,\ldots,a'_{2n}\rangle$. If $a=a'$, then clearly $\langle k_1, \ldots, k_n\rangle=\langle k'_1, \ldots, k'_n\rangle$. Now suppose that $\langle k_1, \ldots, k_n\rangle=\langle k'_1, \ldots, k'_n\rangle$. Then for each $i$, we prove that the two sets of numbers from the $i$th to $i+1$th odd numbers in the respective sequences $a$ and $a'$ are equal. The base case for $i=0$ holds as $a_1=a'_1=1$. Suppose it holds for some $i\geq 0$. Let the $i+1$th odd number in $a$ and $a'$ be $a_j$ and $a'_j$ respectively. Now, we have $k_{i+1} = k'_{i+1}$. If $k_{i+1}=k'_{i+1}=0$, then, by definition, there is no even number between $a_j$ (equivalently $a'_j$) and the next odd number in the sequence. Thus, $a_{j+1}$ (equivalently $a'_{j+1}$) must be odd. Since $a_{j+1}\in \{a_j + 1, a_j+2\}$ and $a'_{j+1}\in \{a'_j + 1, a'_j+2\}$, we must have $a_{j+1}=a_j+2=a'_j+2=a'_{j+1}$, where $a_{j+1}$ and $a'_{j+1}$ are the $i+2$th odd numbers in $a$ and $a'$ respectively. Now consider the case when $k_{i+1}=k'_{i+1}>0$. Then, there are $k_{i+1}$ even numbers in $a$ (resp. $a'$) between $a_j$ (resp. $a'_j$) and the next odd number. These numbers must be $a_j+1, a_j+3, \ldots, a_j+2k_{i+1}-1$ and $a'_j+1, a'_j+3, \ldots, a'_j+2k_{i+1}-1$. The next odd number must then be $a_j+2k_{i+1}$ and $a'_j+2k_{i+1}$ respectively. Since $a_j=a'_j$, these two sets of numbers are equal. We have $a_{j+1}=a'_{j+1}$. Hence, it follows by induction that for all $0\leq i\leq n-1$, the numbers from the $i$th to $i+1$th odd numbers are equal in the sequences $a$ and $a'$. Also since $\langle k_1, \ldots, k_n\rangle=\langle k'_1, \ldots, k'_n\rangle$, we have $a=a'$. 

\vspace{1em}\noindent
Therefore, the number of distinct codes is the number of such sequences $\langle k_1, \ldots, k_n\rangle$. We see that the only constraints on $k_i$ are that they are non-negative and $\displaystyle{\sum_{i=1}^n k_i=n}$. Recall that the number of non-negative integer solutions to the equation $\displaystyle{\sum_{i=1}^r x_i=s}$ is $\binom{s+r-1}{s}=\binom{s+r-1}{r-1}$. Hence, there are $\binom{2n-1}{n}$ such sequences, i.e., $\binom{2n-1}{n}$ distinct codes, and hence there are $\binom{2n-1}{n}$ non-isomorphic acyclic bitournaments with partite sets of size $n$. 

\vspace{1em}\noindent
For the case when $|X|=m$ and $|Y|=n$ with $n\neq m$, $\code(D)=\langle a_1,\ldots,a_{n+m}\rangle$ has $a_1=1$ and either $n$ odd numbers and $m$ even numbers or vice versa. Then, by similar argument as above, the number of distinct codes is the number of sequences $\langle k_1, \ldots, k_n\rangle$ such that each $k_i\geq 0$ and $\displaystyle{\sum_{i=1}^n k_i}=m$ plus the number of sequences $\langle k'_1, \ldots, k'_m\rangle$ such that each $k'_i\geq 0$ and $\displaystyle{\sum_{i=1}^m k'_i=n}$. This is equal to $\binom{m+n-1}{n-1}+\binom{m+n-1}{n}=\binom{m+n}{n}$ (by Pascal's identity). Hence, there are $\binom{m+n}{n}$ non-isomorphic bitournaments in this case.
\end{proof}

\begin{figure}[ht]
\begin{center}
\includegraphics[scale=0.7]{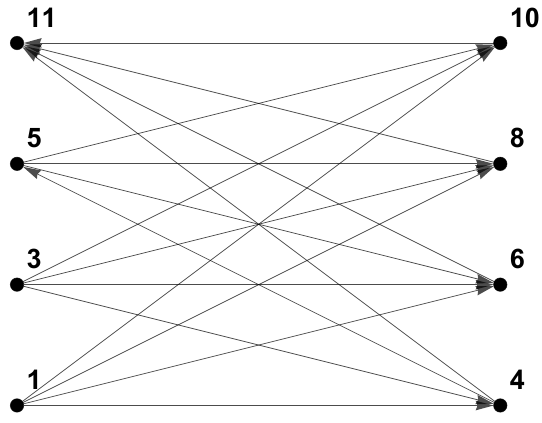}
\caption{An acyclic bitournament with its \texttt{code} $\langle 1,3,4,5,6,8,10,11\rangle$.}\label{graph3}
\end{center}
\end{figure}

\noindent
Finally, we give a linear time algorithm to check isomorphism between two given acyclic bitournaments.
\begin{thm}
There is a linear time algorithm for deciding whether two acyclic bitournaments are isomorphic.
\end{thm}

\begin{proof}
Suppose we are given two acyclic bitournaments $D_1$ and $D_2$ as input and we need to check whether $D_1\cong D_2$. We describe an algorithm. For each of the digraphs, we do the following. Topologically sort it and obtain an ordering $\langle v_1,\ldots,v_n\rangle$ of the vertices. We set labels to the vertices using a function $\ell$ as follows. Set $\ell(v_1)=1$. For $i\geq 2$, if $v_{i-1}$ has an arc to $v_i$, then set $\ell(v_i)=\ell(v_{i-1})+1$. Otherwise, set $\ell(v_i)=\ell(v_{i-1})+2$. Call the sequences $\langle \ell(v_1),\ldots,\ell(v_n) \rangle$ obtained for $D_1$ and $D_2$ as $d_1$ and $d_2$ respectively. We decide that $D_1\cong D_2$ if and only if $d_1=d_2$.

\vspace{1em}\noindent
We first argue the correctness. Let $S$ be the set $\{\ell(v_1),\ldots,\ell(v_n)\}$. Since $D$ is an oriented bipartite graph, we have assigned odd numbers to all vertices in the partite set containing $v_1$ and even numbers to all vertices in the other partite set. Thus for any arc $(v_i,v_j)$, we have $\ell(v_i)\not\equiv_2\ell(v_j)$. Again, because of the topological sorting, if $D$ has an arc $(v_i,v_j)$, then $i<j$. Our construction ensures that if $i<j$, then $\ell(v_i)<\ell(v_j)$. Hence, for each arc $(v_i,v_j)$, we have $\ell(v_i)<\ell(v_j)$ and $\ell(v_i)\not\equiv_2\ell(v_j)$. Therefore, $D\cong D_S$. It follows from Definition \ref{def:scalingfunc} and the above construction that $\Sscl=S$. Therefore, $d_1=\code(D_1)$. By similar argument, $d_2=\code(D_2)$. Thus, by Theorem \ref{bitiso}, we must have $D_1\cong D_2$ if and only if $d_1=d_2$.

\vspace{1em}\noindent
Let us now analyze the runtime. Let $n=|V(D_1)|=|V(D_2)|$ and $m=|E(D_1)|=|E(D_2)|$. For digraphs $D_1$ and $D_2$, each topological sort takes $O(n+m)$. Constructing the labels for each digraph takes $O(n)$ time if the input is given in adjacency matrix form. If the input is of the form of adjacency list, then the construction of labels takes $\displaystyle{\sum_{i=2}^n \text{out-degree}(v_{i-1})=O(m).}$ Finally checking whether the $\code$s are same takes $O(n)$ time. Hence, we get an $O(n+m)$ time, i.e., a linear time algorithm.
\end{proof}

\subsection{Caccetta-H$\ddot{\textrm{a}}$ggkvist Conjecture}\label{sec:ch}

Here, we note that a conjecture for general directed graphs can be solved to some extent for directed bipartite graphs. The Caccetta-H$\ddot{\textrm{a}}$ggkvist Conjecture states: ``Every simple digraph of order $n$ with minimum outdegree at least $r$ has a cycle of length at most $\lceil n/r\rceil$.''
The conjecture has been proved for $r\leq \sqrt{n/2}$ by Shen \cite{shen}. For $r \geq n/2$ it is trivial since that means number of arcs in the graph is at least $n^2/2 > \binom{n}{2}$, which implies the presence of a 2-cycle. But it is still open for $r=n/3,n/4,n/5$ and so on.

\vspace{1em}
\noindent
We consider the conjecture for directed bipartite graphs. For any $r<n$, if there exists a $2$-cycle, we are done. So we can assume that the graphs are oriented bipartite graphs. Let $D=(X,Y,E)$ be an oriented bipartite graph with partite sets $X$ and $Y$, where $|X|=n_1$ and $|Y|=n_2$ ($n_1,n_2\geq 1$), and $E$ is the set of arcs. Let $V = X\cup Y$ be the set of vertices of $D$ with $|V|=n=n_1+n_2$. Consider the conjecture for $r = n/3$. Since an oriented bipartite graph does not have a $3$-cycle, the conjecture implies the following:

\begin{prop}
There exists no oriented bipartite graph of order $n$ with minimum outdegree at least $n/3$.
\end{prop}

\begin{proof}
Suppose $d^{+}(v) \geq n/3 = \frac{n_1+n_2}{3}$ $\forall v\in V$. Then $|E|=\sum\limits_{v\in V} d^{+}(v) \geq \frac{(n_1+n_2)^2}{3} \geq \frac{4n_1n_2}{3} > n_1n_2$ which is a contradiction since $|E| \leq n_1n_2$. 
\end{proof}

\noindent
Now we have the following improvement of the above result.

\begin{prop} 
There exists no oriented bipartite graph of order $n$ with minimum outdegree $> n/4$.
\end{prop}

\begin{proof}
 If $\forall v\in V$, $d^{+}(v) \geq n/4$ and $\exists v_0 \in V$ such that $d^{+}(v_0) > n/4$, then $|E|=\sum\limits_{v\in V} d^{+}(v) > \frac{(n_1+n_2)^2}{4} \geq n_1n_2$ which is again a contradiction as before.
\end{proof}

\noindent
Thus, the above proposition can be restated as the following.

\begin{cor}
In any oriented bipartite graph of order $n$, there exists a vertex with outdegree at most $n/4$.
\end{cor}

\vspace{1em}
\noindent
Now it follows that in an oriented bipartite graph with minimum outdegree $n/4$, every vertex has outdegree exactly $n/4$. Then, $n_1n_2 \geq |E| = n^2/4 =  (n_1+n_2)^2/4 \geq n_1n_2$. Therefore, $|E|= n_1n_2 = (n_1+n_2)^2/4$ and hence, $n_1=n_2$. Thus, we see that $D$ is an oriented \textit{complete} bipartite graph, i.e., a bitournament with $|X|=|Y|$. Note that since $n/4=n_1/2$ is an integer (the exact outdegree of each vertex), $n_1$ must be even. Since the underlying undirected bipartite graph is complete, the in-degree of each vertex must also be $n_1/2$. Therefore, for $r=n/4$, the Caccetta-H$\ddot{\textrm{a}}$ggkvist conjecture for oriented bipartite digraphs can be restated as the following.

\begin{thm}\label{tn4}
Let $D=(X,Y,E)$ be a bitournament with $|X|=|Y|=2m$ and $d^{+}(v)=d^{-}(v)=m$ $\forall v \in V=X\cup Y$. Then $D$ contains a 4-cycle.
\end{thm}

\begin{proof}
Consider any 2-path $u \rightarrow v \rightarrow w$ in $D$ where $u,w \in X$ and $v \in Y$. Let $N(w) \subset Y$ be the set of $m$ out-neighbors of $w$. All vertices in $N(w)$ cannot be out-neighbors of $u$, otherwise $d^+(u)\geq |N(w)\cup \{v\}| = m+1$ which is a contradiction. Hence $\exists x \in N(w)$ such that $x$ is not an out-neighbor of $u$ and hence an in-neighbor of $u$. (Since every vertex in $Y$ is either an in-neighbor or an out-neighbor of $u$). Thus we have the 4-cycle $u \rightarrow v \rightarrow w \rightarrow x \rightarrow u$. 
\end{proof}

\begin{rem}
Note that in the case of Theorem \ref{tn4}, $D$ cannot be bitransitive by Theorem \ref{bitbit}.
\end{rem}

\noindent
We now prove the conjecture for the case $r=n/5$. Since a bipartite graph cannot have a 5-cycle, the case for $r=n/5$ can be restated as: ``An oriented bipartite graph $(X_0,X_1,E)$ with $|X_0 \cup X_1| = n$ and minimum out-degree at least $n/5$ has a directed $4$-cycle.''
 
\noindent
We use some notations: for $i \geq 0$, let $N_{i}(v)$ denote the $i$th neighborhood of a vertex $v$, i.e., the set of vertices which are at distance $i$ from $v$, and let $N_{-1}(v)$ denote the set of in-neighbors of the vertex $v$.

\begin{lem}\label{1stlem}
In a bipartite graph with bipartition $(X_0, X_1)$ and minimum out-degree at least $n/5$, if for some $i\in \{0,1\}$, $|X_{i}|\le \alpha_1 n$ and $|X_{1-i}| \ge \alpha_2 n$, then there exists $v \in X_{i}$ such that $|N_{-1}(v)| \geq \dfrac{\alpha_2}{5\alpha_1}n$.
\end{lem}

\begin{proof}
Since minimum out-degree of a vertex is at least $n/5$, there are at least $\alpha_2 n^2/5$ outgoing arcs from $X_{1-i}$, which are ``received'' by at most $\alpha_1n$ vertices in $X_i$. Hence, by pigeon-hole principle, there exists a vertex $v \in X_i$ which ``receives'' at least $\dfrac{\alpha_2n^2/5}{\alpha_1n}$ many arcs. Thus, $|N_{-1}(v)| \geq \dfrac{\alpha_2}{5\alpha_1}n$.
\end{proof}

\begin{lem}\label{2ndlem}
Let $G=(X_0, X_1, E)$ be an oriented bipartite graph that does not contain a $4$-cycle and has minimum out-degree at least $n/5$. Let $v$ be a vertex in $X_i$, for some $i\in \{0,1\}$, such that $|N_{-1}(v)| \geq \alpha n$. Then 

\begin{enumerate}
\item[(i)] $|N_1(v)\cup N_3(v)| \leq |X_{1-i}| - \alpha n$
\item[(ii)] $|N_2(v)| \geq \dfrac{0.04n}{\frac{|X_{1-i}|}{n} - \alpha - 0.2}$
\end{enumerate}
\end{lem}

\begin{proof} 
Note that if $N_3(v) \cap N_{-1}(v) \neq \emptyset$, then there is a directed $4$-cycle, which is a contradiction. Since the graph is oriented, we also have $N_1(v) \cap N_{-1}(v) \neq \emptyset$. Thus, $(N_1(v)\cup N_3(v)) \cap N_{-1}(v) = \emptyset$. Hence, $|N_1(v)\cup N_3(v)| \leq |X_{1-i}| - |N_{-1}(v)| \leq |X_{1-i}| - \alpha n$, which proves (i).

\vspace{1em}
\noindent
Now, consider the graph $G'$ induced by $N_2(v) \cup (N_1(v) \cup N_3(v))$. Since it is oriented, the number of arcs in $G'$ is at most $|N_2(v)||N_1(v) \cup N_3(v)| \leq |N_2(v)| (|X_{1-i}| - \alpha n)$. Again, the number of arcs in $G'$ is at least the number of arcs ``exiting'' $N_2(v)$ and $N_1(v)$, which is at least $(|N_1(v)|+|N_2(v)|)\dfrac{n}{5} \geq |N_2(v)|\dfrac{n}{5} + \dfrac{n^2}{25}$. Thus, we get the inequality $|N_2(v)|\dfrac{n}{5} + \dfrac{n^2}{25} \leq |N_2(v)|(|X_{1-i}| - \alpha n)$, which gives $|N_2(v)|\geq \dfrac{0.04n}{\frac{|X_{1-i}|}{n} - \alpha - 0.2}$, and this proves (ii).
\end{proof}

\noindent
Now we invoke Lemma \ref{1stlem} and Lemma \ref{2ndlem} repeatedly to prove the following theorem:

\begin{thm}\label{n5thm}
An oriented bipartite graph $G=(X_0,X_1,E)$ with $|X_0 \cup X_1| = n$ and minimum out-degree at least $n/5$ has a directed $4$-cycle.
\end{thm}

\begin{proof}
Assume that $G$ does not contain a directed $4$-cycle. WLOG, let $|X_0|\leq |X_1|$. We prove the theorem by considering the following cases:

\vspace{1em}
\noindent
\textbf{Case 1.} $|X_1|\geq 0.75n$.

\noindent
Note that $|X_0| \leq 0.25n$. By Lemma \ref{1stlem}, $\exists v\in X_0$ such that $|N_{-1}(v)| \geq 0.6n$. Now, $|N_1(v)| \geq n/5$ and $N_1(v) \cap N_{-1}(v) = \emptyset$ since the graph is oriented. Again, $|N_2(v)| \geq n/5$ and $v\not\in N_2(v)$. Thus $|V| \geq |N_{-1}(v)\cup N_1(v)\cup N_2(v) \cup\{v\}| \geq 3n/5 + n/5 + n/5 + 1 = n+1$, which is a contradiction.

\vspace{1em}
\noindent
\textbf{Case 2.} $0.65n \leq |X_1| < 0.75n$.

\noindent
Note that $0.25n < |X_0| \leq 0.35n$. By Lemma \ref{1stlem}, $\exists v\in X_0$ such that $|N_{-1}(v)| > 0.371n$. Then, by Lemma \ref{2ndlem}, $|N_1(v)\cup N_3(v)| < 0.379n$ and $|N_2(v)| > 0.223n$. Again, by applying Lemma \ref{1stlem} on the induced bipartite graph with bipartition $(N_1(v)\cup N_3(v), N_2(v))$, $\exists u \in N_1(v)\cup N_3(v) \subset X_1$ such that $|N_{-1}(u)| > 0.117n$. Then, by Lemma \ref{2ndlem}, $|N_2(u)| > \frac{0.04n}{0.033} > n$, which is a contradiction.

\vspace{1em}
\noindent
\textbf{Case 3.} $0.6n \leq |X_1| < 0.65n$.

\noindent
Note that $0.35n < |X_0| \leq 0.4n$. By Lemma \ref{1stlem}, $\exists v\in X_0$ such that $|N_{-1}(v)| \geq 0.3n$. Then, by Lemma \ref{2ndlem}, $|N_1(v)\cup N_3(v)| \leq 0.35n$ and $|N_2(v)|> 0.26n$. Again, by applying Lemma \ref{1stlem} on the induced bipartite graph with bipartition $(N_1(v)\cup N_3(v), N_2(v))$, $\exists u \in N_1(v)\cup N_3(v) \subset X_1$ such that $|N_{-1}(u)| > 0.14n$. Then, by Lemma \ref{2ndlem}, $|N_2(u)| > 0.66$n, which is a contradiction since $N_2(u) \subset X_1$ and $|X_1|< 0.65n$.

\vspace{1em}
\noindent
\textbf{Case 4.} $0.56n \leq |X_1| < 0.6n$.

\noindent
Note that $0.4n < |X_0| \leq 0.44n$. By Lemma \ref{1stlem}, $\exists v\in X_0$ such that $|N_{-1}(v)| > 0.254n$. Then, by Lemma \ref{2ndlem}, $|N_1(v)\cup N_3(v)| < 0.35n$ and $|N_2(v)|> 0.27n$. Again, by applying Lemma \ref{1stlem} on the induced bipartite graph with bipartition $(N_1(v)\cup N_3(v), N_2(v))$, $\exists u \in N_1(v)\cup N_3(v) \subset X_1$ such that $|N_{-1}(u)| > 0.154n$. Then, by Lemma \ref{2ndlem}, $|N_1(v)\cup N_3(v)| < 0.286n$ and $|N_2(u)| > 0.465$n. By applying Lemma \ref{1stlem} on the induced bipartite graph with bipartition $(N_1(u)\cup N_3(u), N_2(u))$, $\exists w \in N_1(u)\cup N_3(u) \subset X_0$ such that $|N_{-1}(w)| > 0.339n$. Then, by Lemma \ref{2ndlem}, $|N_2(w)| > 0.655$n, which is a contradiction since $N_2(w) \subset X_0$ and $|X_0|\leq 0.44n$.

\vspace{1em}
\noindent
\textbf{Case 5.} $0.53n \leq |X_1| < 0.56n$.

\noindent
Note that $0.44n < |X_0| \leq 0.47n$. By Lemma \ref{1stlem}, $\exists v\in X_0$ such that $|N_{-1}(v)| > 0.225n$. Then, by Lemma \ref{2ndlem}, $|N_1(v)\cup N_3(v)| < 0.335n$ and $|N_2(v)|> 0.296n$. Again, by applying Lemma \ref{1stlem} on the induced bipartite graph with bipartition $(N_1(v)\cup N_3(v), N_2(v))$, $\exists u \in N_1(v)\cup N_3(v) \subset X_1$ such that $|N_{-1}(u)| > 0.176n$. Then, by Lemma \ref{2ndlem}, $|N_1(u)\cup N_3(u)| < 0.294n$ and $|N_2(u)| > 0.425n$. By applying Lemma \ref{1stlem} on the induced bipartite graph with bipartition $(N_1(u)\cup N_3(u), N_2(u))$, $\exists w \in N_1(u)\cup N_3(u) \subset X_0$ such that $|N_{-1}(w)| > 0.289n$. Then, by Lemma \ref{2ndlem}, $|N_2(w)| > 0.563n$, which is a contradiction since $N_2(w) \subset X_0$ and $|X_0|\leq 0.47n$.

\vspace{1em}
\noindent
\textbf{Case 6.} $0.5n \leq |X_1| < 0.53n$.

\noindent
Note that $0.47n < |X_0| \leq 0.5n$. By Lemma \ref{1stlem}, $\exists v\in X_0$ such that $|N_{-1}(v)| \geq 0.2n$. Then, by Lemma \ref{2ndlem}, $|N_1(v)\cup N_3(v)| < 0.33n$ and $|N_2(v)|> 0.307n$. Again, by applying Lemma \ref{1stlem} on the induced bipartite graph with bipartition $(N_1(v)\cup N_3(v), N_2(v))$, $\exists u \in N_1(v)\cup N_3(v) \subset X_1$ such that $|N_{-1}(u)| > 0.186n$. Then, by Lemma \ref{2ndlem}, $|N_1(u)\cup N_3(u)| < 0.314n$ and $|N_2(u)| > 0.35n$. By applying Lemma \ref{1stlem} on the induced bipartite graph with bipartition $(N_1(u)\cup N_3(u), N_2(u))$, $\exists w \in N_1(u)\cup N_3(u) \subset X_0$ such that $|N_{-1}(w)| > 0.222n$. Then, by Lemma \ref{2ndlem}, $|N_1(w)\cup N_3(w)| < 0.308n$ and $|N_2(w)| > 0.37n$. By applying Lemma \ref{1stlem} on the induced bipartite graph with bipartition $(N_1(w)\cup N_3(w), N_2(w))$, $\exists x \in N_1(w)\cup N_3(w) \subset X_1$ such that $|N_{-1}(x)| > 0.24n$. Then, by Lemma \ref{2ndlem}, $|N_2(x)| > 0.66n$, which is a contradiction since $N_2(x) \subset X_1$ and $|X_1|< 0.53n$.

\vspace{1em}
\noindent
Hence, in each case we get a contradiction, but one of them must hold. Thus, our assumption that there is no $4$-cycle in $G$ must be wrong. This completes the proof.
\end{proof}

\section{Odd-even Graphs}
In this section, we introduce a family of graphs that we call {\em odd-even} graphs. Throughout the section we denote the set of all non-negative even numbers by $\mathcal{E}$ and the set of all positive odd numbers by $\mathcal{O}$. We begin with the definition of oriented odd-even graphs.

\begin{defn}\label{oddeven}
Let $A \subseteq \mathcal{E}$ and $O \subseteq \mathcal{O}$. An \textit{oriented odd-even} graph $\overrightarrow{\mathcal{G}}_A(O)$ is an oriented graph with the set of vertices $A$ and with set of arcs $E = \Set{ab}{\frac{a+b}{2}, \frac{b-a}{2} \in O}$. 
\end{defn}

\noindent
Observe that $\overrightarrow{\mathcal{G}}_A(O)$ is an oriented bipartite graph with  partite sets $V_{1} = \Set{v \in A}{\Mod{v}{0}{4}}$ and $V_{2} = \Set{v \in A}{\Mod{v}{2}{4}}$ as both $\frac{a+b}{2}$ and $\frac{b-a}{2}$ are even for any pair of $a,b \in V_{i}$ and for  each $i \in \set{1,2}$. 

\begin{defn}\label{oegraph}
An \textit{odd-even} graph $\mathcal{G}_A(O)$ is the underlying (undirected) graph of $\overrightarrow{\mathcal{G}}_A(O)$, i.e., $\mathcal{G}_A(O)$ is a graph with set of vertices $A$ and with set of arcs $E = \Set{ab}{\frac{a+b}{2}, \frac{|a-b|}{2} \in O}$. 
\end{defn}

\noindent
From above, it is clear that $\mathcal{G}_A(O)$ is bipartite graph. Interestingly, the following theorem shows that every bipartite graph can be represented by an odd-even graph.

\begin{thm}\label{charbi}
Let $B$ be a bipartite graph. Then there exist $A\subseteq\mathcal{E}$ and $O\subseteq\mathcal{O}$ such that $\mathcal{G}_A(O)$ is isomorphic to $B$.
\end{thm}

\begin{proof} 
Let $B=(X,Y,E)$ be a bipartite graph with the partite sets $X$ and $Y$. Let $X = \set{b_0, b_2,\ldots,b_{2m}}$, $Y = \set{b_1, b_3,\ldots,b_{2n-1}}$ and $V=X\cup Y$. Now define a function $\Map{f}{V}{\mathcal{E}}$ with $f(b_i) = 10^{i+2} +1 + (-1)^{i + 1}$. It is easy to check that the function $f$ is well-defined and injective. Take the even set $A$ to be the image of $f$ and let the odd set $O = \Set{\frac{f(a)+f(b)}{2}, \frac{|f(a)-f(b)|}{2}}{ab \in E(B)}$. Now to show that $B$ is isomorphic to $\mathcal{G}_A(O)$  it is enough to observe that $f(x) + f(y) \neq f(x^\prime) + f(y^\prime)$, $f(x) + f(y) \neq |f(x^\prime) - f(y^\prime)|$, $|f(x) - f(y)| \neq f(x^\prime) + f(y^\prime)$ and $|f(x) - f(y)| \neq |f(x^\prime) - f(y^\prime)|$ for any $xy \in E(B)$ and $x^\prime y^\prime \notin E(B)$. 
\end{proof}

\noindent
Now from Definition \ref{oddeven} it is clear that $\overrightarrow{\mathcal{G}}_A(O)$ is acyclic, i.e., there is no directed cycle in $\overrightarrow{\mathcal{G}}_A(O)=(V,E)$ as for any arc $ab\in E$, $b-a>0$ and so $a<b$. Thus no come back to the starting vertex is possible in a directed walk. Therefore $\overrightarrow{\mathcal{G}}_A(O)$ is an acyclic oriented bipartite graph. In the following we will see that any acyclic oriented bipartite graph can be represented by an oriented odd-even graph. Let $D=(V,E)$ be a digraph. An ordering $u_1,u_2,\ldots,u_n$ of vertices of $D$ is a {\em topological ordering} (or, {\em acyclic ordering} \cite{JEN}) if for every arc $u_iu_j\in E$, we have $i<j$. It is known that every acyclic digraph has a topological ordering of vertices \cite{JEN}.

\begin{thm}\label{charbi2}
Let $B$ be an acyclic oriented bipartite graph. Then there exist $A\subseteq\mathcal{E}$ and $O\subseteq\mathcal{O}$ such that $\overrightarrow{\mathcal{G}}_A(O)$ is isomorphic to $B$.
\end{thm}

\begin{proof}
Let $B=(X,Y,E)$ be an acyclic oriented bipartite graph with the partite sets $X$ and $Y$ and $V=X\cup Y$. Let $u_1,u_2,\ldots,u_n$ be a topological ordering of $V$. Now we define a function $\Map{f}{V}{\mathcal{E}}$ inductively. Assign $f(u_1)=10^2$ or $10^3+2$ according as $u_1\in X$ or $u_1\in Y$. Suppose $f(u_i)=10^{2k}$ for some $k\in\Nat$. We assign $f(u_{i+1})=10^{2k+2}$ or $10^{2k+1}+2$ according as $u_{i+1}\in X$ or $u_{i+1}\in Y$. If $f(u_i)=10^{2k+1}+2$ for some $k\in\Nat$. We assign $f(u_{i+1})=10^{2k+2}$ or $10^{2k+3}+2$ according as $u_{i+1}\in X$ or $u_{i+1}\in Y$. The function $f$ is well-defined and is a strictly increasing function (hence injective). We take the even set $A$ to be the image of $f$ and let the odd set $O = \Set{\frac{f(a)+f(b)}{2}, \frac{f(b)-f(a)}{2}}{ab \in E(B)}$. Since $f$ is increasing, $f(a)<f(b)$ for all $ab\in E$ due to the topological ordering. Then it follows that $\overrightarrow{\mathcal{G}}_A(O)$ is isomorphic to $B$ (rest of the proof is similar to the proof of Theorem  \ref{charbi}).
\end{proof}

\noindent
Note that the above theorem can easily be extended to bipartite graphs with countably infinite number of vertices. Therefore, the family of  odd-even graphs is, in fact, the family of all bipartite graphs with countable number of vertices. Now we will prove some conditions for finite odd even graphs to be connected. For any odd-even graph $\mathcal{G}_A(O)$, let the \textit{relevant odd set} be  $O_{rel} = O \cap \{\frac{a+b}{2}, \frac{|a-b|}{2} \mid ab \in E \}$. Note that $\mathcal{G}_A(O)$ is isomorphic to $\mathcal{G}_A(O_{rel})$.

\begin{thm}\label{th con1}
If $\mathcal{G}_A(O)$ is connected with $|A| \geq 2$, then $\mid O_{rel} \mid \geq \sqrt{2\mid A\mid}-1$.
\end{thm} 

\begin{proof}
Suppose $|A| = n$ and $|O_{rel}| = k$. Now, the number of edges in $\mathcal{G}_A(O)$ is at least $n-1$ (since $\mathcal{G}_A(O)$ is connected) and at most $\binom{k}{2}+k=\frac{k(k+1)}{2}$. This is because each edge $ab$ corresponds to either a pair of numbers $(\frac{a+b}{2}, \frac{|a-b|}{2}) \in O_{rel}$ or a single number $\frac{a}{2}\in O_{rel}$ (in case $b=0$). Thus, 

\begin{align}\label{eqn face equal}\nonumber
 \frac{k(k+1)}{2} \geq n-1  &\Rightarrow \left(k - \frac{\sqrt{8n - 7}-1}{2}\right)\left(k + \frac{\sqrt{8n-7}+1}{2}\right) \geq 0\\ \nonumber
&\Rightarrow k - \frac{\sqrt{8n - 7}-1}{2} \geq 0 &&(\text{since } k + \frac{\sqrt{8n-7}+1}{2} > 0)\\ \nonumber
&\Rightarrow k \geq \frac{\sqrt{8n - 7}-1}{2} \geq \sqrt{2n} -1 &&(\text{ for } n \geq 2)\\  \nonumber
&\Rightarrow |O_{rel}| \geq \sqrt{2\mid A\mid}-1.\nonumber
\end{align}
\end{proof} 

\begin{thm}\label{th con2}
Suppose $A = \left\{ 0,2,4,...,2(m-1) \right\}$. If $| O_{rel} | \geq \frac{7|A|}{8}$, then $\mathcal{G}_A(O)$ is connected.
\end{thm} 

\begin{proof} WLOG, we can remove the isolated vertices from $A$ and prove the statement for the resulting set $A$. Since the size of $A$ can only decrease, the lower bound of $\frac{7|A|}{8}$ on $O_{rel}$ still holds. 

\vspace{1em}\noindent
Assume to the contrary that $\mathcal{G}_A(O)$ is disconnected. Therefore, there exist at least two connected components. Let $X$ be a connected component and $Y$ be the union of the other connected components. 
Call the larger of these two sets as $Z$ and let $W$ be the other set.  
Then, we have $|Z|\geq \frac{|A|}{2}$. Now, $Z$ can be partitioned into two sets: those of the form $4k+2$ and those of the form $4k$. Let the larger set be $Z'$. We must have $|Z'|\geq \frac{|Z|}{2}\geq \frac{|A|}{4}$.   
Fix a vertex $a \in W$ such that $a$ has a form opposite to that of the numbers in $Z'$. Note that such a number exists in $W$ because if $W$ has all numbers of the same form, then $W$ cannot have edges within itself, and hence, would be a set of isolated vertices. But we removed all such vertices. 

\vspace{1em}\noindent
For $b \in Z'$, define $S_b = \left\{ \frac{(a+b)}{2}, \frac{|a - b|}{2} \right\}$. Note that all elements in $S_b$ are odd. Define \[T = \Set{t}{t \in S_b \text{ for some $b$ and } t \notin O_{rel}}\] We must have $O_{rel}\subseteq \{1,3,\ldots,2m-3\}\setminus T$, where $T\subseteq \{1,3,\ldots,2m-3\}$. 

\vspace{1em}\noindent
Observe that $a$ does not share edges with any vertex in $Z'$. Therefore, at least one element from each $S_b$ does not belong to $O_{rel}$.  Consider the multiset obtained by adding the elements from each $S_b$ that do not belong to $O_{rel}$. It has size at least $|Z'|\geq \frac{|A|}{4}$, and the multiplicity of any number in this multiset can be at most $2$. Hence, removing duplicates, we get the set $T$, which has size at least $\frac{|A|}{8}$. Therefore, $|O_{rel}|\leq |\{1,3,\ldots,2m-3\}| - |T| \leq m-1-\frac{|A|}{8}=|A|-1-\frac{|A|}{8} < \frac{7|A|}{8}$. But this is a contradiction to our premise. Hence, $\mathcal{G}_A(O)$ must be connected.   
\end{proof}

\noindent
Now we study odd-even graphs with odd sets of the following form: 
$$O_{a,b}=\Set{ak+b}{a\in\mathcal{E},b\in\mathcal{O},k\in\Nat}.$$

\begin{thm}\label{th uni1}
The oriented bipartite graph $\overrightarrow{G}=\overrightarrow{\mathcal{G}}_{\mathcal{E}}(O_{a,b})=(V,E)$ is unidirectional if and only if $4$ divides $a$. 
\end{thm} 

\begin{proof}
Let $V_{1} = \{ v \in \mathcal{E} \mid v \equiv 0$ (mod $4)\}$ and $V_{2} = \{ v \in \mathcal{E} \mid v \equiv 2$ (mod $4)\}$. Then $V=V_1\cup V_2$. First assume that $a$ is divisible by $4$. Let $u = 4x \in V_1$, $v = 4y +2 \in V_2$ and $vu \in E$. So that forces $u > v$ as we have $\frac{u+v}{2}, \frac{u-v}{2}  \in O_{a,b}$. That is, we have $2(x+y) + 1,2(x-y) - 1  \in O_{a,b}$. This implies $$2x = a(n_1 +  n_2)/2 + b$$ where $n_1, n_2$ are some positive integers. But this is a contradiction as  $a(n_1 +  n_2)/2 + b$ is an odd number while $2x$ is even. So all the arcs in $\overrightarrow{G}$ are from $V_1$ to $V_2$, i.e., $\overrightarrow{G}$ is unidirectional.

\vspace{1em}
\noindent
For the converse part, assume that $a$ is not divisible by 4. Let $n_1 > n_2$ be two positive even integers. Then $u = a(n_1 - n_2) \in V_1$ and $v = a(n_1+n_2) +2b \in V_2$. In this case, $\frac{u+v}{2}, \frac{v-u}{2}  \in O_{a,b}$ and we have the arc $uv \in E$. On the other hand, consider two positive integers $m_1 > m_2$ where $m_1$ is odd and $m_2$ is even. Then $u' = a(m_1+m_2) +2b \in V_1$ and $v' = a(m_1 - m_2) \in V_2$. In this case, $\frac{u^\prime+v^\prime}{2}, \frac{u^\prime-v^\prime}{2}  \in O_{a,b}$ and we have the arc $v'u' \in E$. So the graph $\overrightarrow{G}$ is not unidirectional when $a$ is not divisible by 4. 
\end{proof} 

\begin{thm}\label{th uni2}
Let $\emptyset\neq I\subset\Nat$ and the odd set is given by $O = \Set{a_i + 1}{a_i\in\mathcal{E}, i \in I}$. Then the oriented graph $\overrightarrow{\mathcal{G}}_{\mathcal{E}}(O)$ is unidirectional if and only if $4$ divides $a_i$ for all $i \in I$ or, $4$ does not divide $a_i$ for all $i \in I$.
\end{thm} 

\begin{proof}\ Let $V_{1} = \{ v \in \mathcal{E} \mid v \equiv 0$ (mod $4)\}$, $V_{2} = \{ v \in \mathcal{E} \mid v \equiv 2$ (mod $4)\}$, $V=V_1\cup V_2$ and $\overrightarrow{\mathcal{G}}_{\mathcal{E}}(O)=(V,E)$. First suppose that $4$ divides $a_i$ for all $i \in I$. Let $u,v\in\mathcal{E}$ such that $uv\in E$. Then $\frac{u+v}{2},\frac{v-u}{2}\in O$. Thus $\frac{u+v}{2}=4k_1+1$ and $\frac{v-u}{2}=4k_2+1$ for some $k_1,k_2\in\Nat$. These imply $v=4(k_1+k_2)+2$ and $u=4(k_1-k_2)$. Thus all the arcs are from $V_1$ to $V_2$.

\vspace{1em}\noindent
Next consider that $4$ does not divide $a_i$ for all $i \in I$. Let $u,v\in\mathcal{E}$ such that $uv\in E$. Then as before we have $\frac{u+v}{2}=4k_1+3$ and $\frac{v-u}{2}=4k_2+3$ for some $k_1,k_2\in\Nat$. These imply $v=4(k_1+k_2+1)+2$ and $u=4(k_1-k_2)$. Thus again we have all the arcs are from $V_1$ to $V_2$.

\vspace{1em}\noindent
Finally let $a_1+1,a_2+1\in O$ such that $4$ divides $a_1$ and $4$ does not divide $a_2$. Let $a_1>a_2$. Consider $u=a_1-a_2$ and $v=a_1+a_2+2$. Then $u\in V_2$, $v\in V_1$, $\frac{u+v}{2}=a_1+1$ and $\frac{v-u}{2}=a_2+1$. Thus $uv\in E$. Again for $w=2a_1+2$, we have $0w\in E$, where $0\in V_1$ and $w\in V_2$. Thus the graph is not unidirectional. For $a_1<a_2$, the proof is similar with the choice $u=a_2-a_1$.
\end{proof}

\noindent
The adjacency matrix of the oriented graph $\overrightarrow{\mathcal{G}}_{\mathcal{E}}(O_{4,1})$ is of the form
$$\begin{array}{|c|c|}
\hline
\mathbf{0} & X\\
\hline
\mathbf{0} & \mathbf{0}\\
\hline
\end{array} \qquad \text{where}\qquad X\ =\ 
\begin{tiny}
\begin{bmatrix}
1 & \quad 0 & \quad 1 & \quad 0 & \quad 1 & \quad 0 & \quad ... \\

0 & \quad 1 & \quad 0 & \quad 1 & \quad 0 & \quad 1 & \quad ... \\

0 & \quad 0 & \quad 1 & \quad 0 & \quad 1 & \quad 0 & \quad ...\\

0 & \quad 0 & \quad 0 & \quad 1 & \quad 0 & \quad 1 & \quad ...\\

0 & \quad ...\\
\end{bmatrix}
\end{tiny}$$

\noindent
and the adjacency matrix of $\overrightarrow{\mathcal{G}}_{\mathcal{E}}(O_{6,1})$ is

\begin{tiny}
\[ \left[ \begin{array}{c|c}
 \mathbf{0} & \begin{bmatrix}
1 &  0 &  0 &  1 &  0 &  0 & 1 &  0 &  0 &  ... \\

0 &  0 &  0 &  0 &  0 &  0 & 0 &  0 &  0 &  ... \\

0 &  0 &  0 &  0 &  0 &  0 & 0 &  0 &  0 &  ... \\

0 &  0 &  0 &  0 &  1 &  0 &  0 &  1 & 0 &  ... \\

0 & \quad ...\\

\end{bmatrix} \\

 \hline
 
 \begin{bmatrix}
0 &  0 &  0 &  0 &  0 &  0 & 0 &  0 &  0 &  ... \\

0 &  1 &  0 &  0 &  1 &  0 & 0 &  1 &  0 &  ... \\

0 &  0 &  0 &  0 &  0 &  0 & 0 &  0 &  0 &  ... \\

0 &  0 &  0 &  0 &  0 &  0 & 0 &  0 &  0 &  ... \\

0 &  0 &  0 &  0 &  1 &  0 & 0 &  1 &  0 &  ... \\

0 & \quad ...\\

\end{bmatrix} & \mathbf{0} \end{array} \right].\]
\end{tiny}

\noindent
Note that according to Theorem~\ref{th uni1}, $\overrightarrow{\mathcal{G}}_{\mathcal{E}}(O_{4,1})$ is unidirectional while $\overrightarrow{\mathcal{G}}_{\mathcal{E}}(O_{6,1})$ is not. From  the above two examples one can observe the difference between the 
adjacency matrices of unidirectional and not unidirectional oriented odd-even graphs. 

\section{The Goldbach Graph}\label{sec gg}

Here, we focus on a particular odd-even graph $\overrightarrow{\mathcal{G}}_{\mathcal{E}}(\mathcal{P})$ and $\mathcal{G}_{\mathcal{E}}(\mathcal{P})$ where the odd set $\mathcal{P}$ is the set of all odd primes, and call them the {\em Goldbach (infinite) digraph} and the {\em Goldbach (infinite) graph} respectively for the reason that will become apparent in the first result of this section. The set of vertices of the Goldbach (infinite) graph is the set of all non-negative even integers and two such vertices $a,b$ are adjacent if and only if both $\frac{a+b}{2}$ and $\frac{|a-b|}{2}$ are odd prime numbers. Let $\mathcal{E}_{n}$ denote the set of all non-negative even numbers less than or equal to $2n$. Also, the graph $\mathcal{G}_{\mathcal{E}_n}(\mathcal{P})$ will be denoted by $\mathcal{G}_{n}$ and we call this graph a {\em Goldbach (finite) graph}. The neighborhood $N_{\mathcal{G}_{n}}(v)$ (or, the out-neighbor $N^+_{\mathcal{G}_{n}}(v)$ or the in-neighbor $N^-_{\mathcal{G}_{n}}(v)$) of a vertex $v$ in $\mathcal{G}_{n}$ (or, in $\overrightarrow{\mathcal{G}}_{n}$ which we call a {\em Goldbach (finite) digraph}) will be denoted by $N_n(v)$ (or $N_n^+(v)$ or $N_n^-(v)$, respectively) for the remainder of the section. Also the degree $d_{\mathcal{G}_{n}}(v)$ (or, the out-degree $d^+_{\mathcal{G}_{n}}(v)$ or the in-degree $d^-_{\mathcal{G}_{n}}(v)$) of a vertex $v$ in $\mathcal{G}_{n}$ (or, in $\overrightarrow{\mathcal{G}}_{n}$) will be denoted by $d_n(v)$ (or $d_n^+(v)$ or $d_n^-(v)$, respectively) for the remainder of the section. We denote $\overrightarrow{\mathcal{G}}_{\mathcal{E}}(\mathcal{P})$ and $\mathcal{G}_{\mathcal{E}}(\mathcal{P})$ by $\overrightarrow{\mathcal{G}}_{\infty}$ and $\mathcal{G}_{\infty}$ respectively and the out-degree and the in-degree of $v\in\mathcal{E}$ in $\overrightarrow{\mathcal{G}}_{\infty}$ by $d_{\infty}^+(v)$ and $d_{\infty}^-(v)$ respectively and the degree of $v\in\mathcal{E}$ in $\mathcal{G}_{\infty}$ by $d_\infty(v)$. $N_\infty^+(v)$ or $N_\infty^-(v)$ are defined similarly. Now we state the result that, by and large,  motivated this work.

\begin{thm}\label{th goldbach} The following statements are equivalent. 
\begin{enumerate}
\item[(i)] (Goldbach's conjecture) Every even integer greater than 5 can be written as sum of two odd primes.
\item[(ii)] $\mathcal{G}_n$ is connected for all $ n \geq 7$.
\item[(iii)] $d_{\infty}^-(v) > 0$ for all $v \geq 6$ in $\overrightarrow{\mathcal{G}}_{\infty}$.  
\end{enumerate}
\end{thm}

\begin{proof}

\noindent
$\mathbf{(i)\Rightarrow(ii)}$:  Suppose that the Goldbach's conjecture is true. Observe that  $\mathcal{G}_7$ is connected. Now assume that
$\mathcal{G}_n$ is connected for all $n \leq k$. By Goldbach Conjecture, $2(k + 1) = p + q$ for some $p, q \in \mathcal{P}$. Then $|p - q|$ is even and $|p - q| < p + q = 2(k + 1)$. Thus $2(k + 1)$ is adjacent to $|p - q|$ which is a vertex of $\mathcal{G}_k$ as well. This implies that $\mathcal{G}_{k+1}$ is connected. 
 
\vspace{1em}
\noindent
$\mathbf{(ii)\Rightarrow(iii)}$:  Suppose $\mathcal{G}_n$ is connected for all $n \geq 7$. Let $v$ be any even integer greater equal to 14. Then, as the graph $\mathcal{G}_{v/2}$ is connected, the vertex $v$ of the graph must be adjacent to some other vertex of the graph. Note that $v$ is the greatest vertex in $\mathcal{G}_{v/2}$. Hence  $d_{\infty}^-(v) > 0$. Now it is a simple observation that for $6 \leq v \leq 12$ we have $d_{\infty}^-(v) > 0$ as $0\in N_\infty^-(6)\cap N_\infty^-(10)$ and $2\in N_\infty^-(8)\cap N_\infty^-(12)$. This completes the proof. 

\vspace{1em}
\noindent
$\mathbf{(iii)\Rightarrow(i)}$: Suppose  $d_{\infty}^-(v) > 0$ for all $v \geq 6$. Now for any even number $a > 5$ there exists $b$ such that $b \in N^-_{\infty}(a)$. That means, there exist odd primes $p,q$ such that we have $p+q = a$. This is precisely the Goldbach's conjecture.
\end{proof}

\noindent
The above result shows that the Goldbach's conjecture can be formulated using graph theoretic notions. Note that in Theorem~\ref{th con1} and~\ref{th con2} we presented one necessary and another sufficient conditions for connectedness of finite odd-even graphs. Improved results of similar nature might give rise to an alternative way of digging into the Goldbach's conjecture using graph theory due to Theorem~\ref{th goldbach}. Having proved this equivalence, naturally we tried to explore more such equivalent formulations. Our observation which was integral in proving the above result is  that, given a non-negative even integer $2n$, it is adjacent to a smaller integer implies that $2n$ can be expressed as the sum of two odd primes. Similarly, its adjacency with a greater integer implies that $2n$ can be expressed as difference of two odd primes.   
This readily provides graph theoretic formulation of another well-known conjecture in number theory. 

\begin{thm} The following statements are equivalent. 

\begin{itemize}
\item[$(i)$] (A  conjecture by Maillet~\cite{mal}) Every non-negative even integer can be written as difference of two odd primes.

\item[$(ii)$] $d_{\infty}^+(v) > 0$ for all $v \geq 2$ in $\overrightarrow{\mathcal{G}_{\infty}}$. 
\end{itemize}
\end{thm}

\noindent
After this the first thing that came to our notice is that the degree of the vertices of our graph is particularly interesting. As the graph is an infinite graph, the natural question about the degrees are, if they are finite or not. In particular, note that each vertex have finite in-degree, as its in-neighbors are smaller non-negative even numbers, while its out-degree can be unbounded. So the vertex $0$ has no in-neighbors while its out-neighbors are precisely $2p$ for all $p \in \mathcal{P}$. 
We know that 
there are infinitely many odd primes due to Euclid's theorem  (which says, there are infinitely many prime numbers). 
Hence, $d^+_{\infty}(0)$ is infinite and this is equivalent to Euclid's theorem. 

\begin{obs}
The vertex $0$ of $\overrightarrow{\mathcal{G}}_{\infty}$ has infinitely many out-neighbors and hence, has infinitely many neighbors.
\end{obs}

\noindent
This observation naturally motivates us to wonder if the degree (or out-degree) of the other vertices are finite or not. It turns out to be a difficult question as it is equivalent to another well-known conjecture, the Kronecker's conjecture. 

\begin{thm} The following statements are equivalent. 

\begin{enumerate}
\item[$(i)$] (Kronecker's conjecture~\cite{KR}) Given a positive even number $2k$, there are infinitely many pairs of primes of the form $\{p,p+2k\}$.
\item[$(ii)$] For every vertex $v \in \mathcal{E}$ we have $d_\infty^+(v)$ is infinite in $\overrightarrow{\mathcal{G}}_{\infty}$.
\item[$(iii)$] For every vertex $v \in \mathcal{E}$ we have $d_\infty(v)$ is infinite in $\mathcal{G}_{\infty}$.
\end{enumerate}
\end{thm}

\begin{proof}

\noindent
$\mathbf{(i)\Rightarrow(ii)}$: Suppose that the conjecture is true. Let $2k$ be an even number for some $k \geq 1$. So, there are infinitely many pairs of primes of the form $\{p,p+2k\}$ by assumption. Note that for each such pair of primes the vertex $2k$ is adjacent to the vertex $2(p+k)$ in $\mathcal{G}_\infty$. 

\vspace{1em}
\noindent 
$\mathbf{(ii) \Leftrightarrow (iii)}$:  Clearly follows from the fact that $ d_\infty^+(v) \leq d_\infty^+(v) + d_\infty^-(v) = d_\infty(v)$ for all $v \in \mathcal{E}$ while $d_\infty^-(v)$ is finite. 

\vspace{1em}
\noindent 
$\mathbf{(iii)\Rightarrow(i)}$: Suppose $d_{\infty}^+(v)$ is infinite  for all $v \in \mathcal{E}$. Let $v= 2k$ be an even number for some $k \geq 1$. Now for each out-neighbor $u=2n$ of $v$ in  $\mathcal{G}_\infty$ we have $\frac{2n+2k}{2}, \frac{(2n-2k)}{2} \in \mathcal{P}$. Hence, both $(n-k)$ and $(n+k)$ are primes and there are infinitely such distincts pairs for each $k \geq 1$.  
\end{proof}

\noindent
In particular, determining if degree (or out-degree) of $2$ is finite or not will settle the twin prime conjecture~\cite{zhang} (positively if $d(2)$ is infinite). This implies an immediate corollary.

\begin{cor} The following statements are equivalent. 
\begin{enumerate}
\item (Twin prime conjecture~\cite{zhang}) There are infinitely many pairs of primes of the form $\{p,p+2\}$.
\item $d_\infty^+(2)$ is infinite in $\overrightarrow{\mathcal{G}}_{\infty}$.
\item $d_\infty(2)$ is infinite in $\mathcal{G}_{\infty}$.
\end{enumerate}
\end{cor}

\noindent
Next we will try to understand the significance of  the degrees of the vertices in $\mathcal{G}_{\infty}$. 
Given a non-negative even number $2n$, the in-degree $d_\infty^-(2n)$  is the number of ways $2n$ can be expressed as the sum of two odd 
primes. Similarly, the out-degree $d_\infty^+(2n)$ is the number of ways $2n$ can be expressed as the difference of two odd 
primes. Moreover, the degree of $0$ in $\mathcal{G}_{n}$ is the number of odd primes less than or equal to $n$. 
So, the graph parameter $d_n(0)$ can be regarded as a function similar to the prime counting function $\pi(n)$, which 
denotes the number of primes less than or equal to $n$. So, for $n \geq 2$ we have $$\pi(n) = d_n(0) +1$$ as the only even prime $2$ is not adjacent to $0$. As it turned out to be an  interesting yet difficult problem to figure out what the degrees of the vertices are, we started to
establish some relations between them. Hence the following result.   
 
\begin{thm}
For all $n \geq 2r$ and for $0 \leq m \leq 4$, in $\overrightarrow{\mathcal{G}}_{\infty}$ we have 
$$\sum_{i=0}^{m} d_n^+(2i) \geq \sum_{i=0}^{m} d_n^-(2r-2i).$$ 
\end{thm} 

\noindent \textit{Sketch of the proof.}
Let $A_i = \left\{ q \mid p + q = 2r - 2i \text{ and } q \leq p \right\}$ for $i \in \{0,1,2,3,4\}$ where $p,q$ are odd primes.
Observe that $d_n^-(2r-2i) = |A_i|$.
Note that for any $q \in \bigcup_{i=0}^4 A_i$, we have 
$2q \in N_n^+(0).$
Thus, $d_n^+(0) \geq \mid \bigcup_{i=0}^4 A_i \mid$. 

\noindent
Now suppose $q \in A_i \cap A_j$ for some $i,j \in \{0,1,2,3,4\}$ and $i <j$. 
Then there are primes $p_1,p_2 \geq q$ such that $p_1 +q = 2r-2i$ and $p_2 +q = 2r-2j$. So $p_2 = p_1 - 2(j-i)$.
As both $\frac{2p_1-2(j-i)+2(j-i)}{2} = p_1$ and $\frac{2p_1-2(j-i)-2(j-i)}{2} = p_1-2(j-i)=p_2$ are odd primes
we have  

\begin{equation}\label{eqn2}
2p_1 - 2(j-i)) \in N_n^+(2(j-i)).
\end{equation}
 
\noindent
Let $S_i = \{(i,x) \mid x \in N_n^+(2i)\}$ for $i \in \{0,1,2,3,4\}$. 
Note that $S_i \cap S_j = \emptyset$ for $i \neq j$ and $|S_i| = d_n^+(2i)$ for all $i \in \{0,1,2,3,4\}$. 
Also let $S  = \bigcup_{i=0}^4 S_i$. 
Then we will construct a subset $T \subseteq S$ such that $|T| \geq \sum_{i=0}^4 \mid A_{i} \mid$. This will complete the proof.

\noindent
\textbf{Step 0:} We know that for each $q \in \bigcup_{i=0}^4 A_i$, we have $(0,2q) \in S_0$.
Put all these $(0,2q)$'s in the set $T$. 
Next we have to deal with the elements that are in more than one $A_i$'s. 

\noindent
\textbf{Step 1:} First we handle the case where an element $q \in A_0 \cap A_j$ 
for some $j \in \{1,2,3,4\}$. For each such $q$ there is a prime $p$ such that $p+q = 2r$. By (\ref{eqn2}) we know that
that for each such $q$, there is an edge between $2j$ and $(2p - 2j)$.  Put all these $(j,2p-2j)$'s in $T$. Observe that all these are new elements in $T$ as $j \geq 1$. 

\noindent
\textbf{Step 2:} Now consider an element $q \in A_1 \cap A_2$. Then there exists a prime  $p$ such that $p+q = 2r - 2$ and by (\ref{eqn2}) we know that 
$(1,2p-2) \in S_1$. 
We will put all such $(1,2p-2)$'s in $T$ if they were already not in $T$. 
Let for some $q$, its corresponding $(1,2p-2)$ were already in $T$. That means that element was included to $T$ due to Step~1.
Therefore, $p +(q+2) = 2r$ where $(q+2)$ is also a prime. Hence, $(1,2q +2) \in S_1$ as both $\frac{2q+2+2}{2}=(q+2)$ and $\frac{2q+2-2}{2}=q$     are primes. 
Note that, all the elements included to $T$ before are of the form $(1,2p-2)$ with $p \notin A_0$ 
while $(q+2) \in A_0$. Therefore,  $(1,2q +2)$ is not yet included to $T$. Now we include all such  $(1,2q +2)$'s to $T$.

\noindent
\textbf{Step 3:} Now consider an element $q \in A_1 \cap A_3$. Then there exists a prime  $p$ such that $p+q = 2r - 2$ and by (\ref{eqn2}) we know that 
$(2,2p-4) \in S_2$. 
We will put all such $(2,2p-4)$'s in $T$ if they were not already in $T$.

\noindent
Let for some $q$, its corresponding $(2,2p-4)$ were already in $T$. That means that element was included to $T$ due to Step~1.
An argument similar to Step~2 will show that there is an edge between $(2q+2)$ and $2$. We will include all those $(1,2q+2)$'s to $T$ which were not 
included to $T$ before. 

\noindent
There may be some $(1,2q+2)$ which was included to $T$ before. Then that inclusion was due to Step~2. 
This implies $p$, $(p-2)$ and $(p-4)$ are all odd primes. The only such instance is when $p=7$. Thus, $2r -6 = (p-4) +q =3+q$. As $(p-4) \geq q$ we have $q =3$. Hence, $2r = 12$. It is easy to check that the theorem holds for $2r=12$. Therefore, we can ignore this case. 

\noindent
There are four more steps, namely, for $q \in A_1 \cap A_4$, $q \in A_2 \cap A_3$, $q \in A_2 \cap A_4$ and $q\in A_3\cap A_4$ in that order, that will conclude the proof. Those cases can be handled in a similar way like above.
\hfill  $\square $

\medskip

\noindent
Our interest in the degree of the bipartite graph $\mathcal{G}_\infty$ prompted us to study the complete bipartite subgraphs of $\mathcal{G}_\infty$ from number theoretic point of view. 
    
\begin{prop}
If the complete bipartite graph $K_{m,n}$ is a subgraph of $\mathcal{G}_\infty$, then there exists a set $\{p_1, p_2, \ldots, p_m\}$ of $m$ primes and a set $\{r_1, r_2, ... r_{n-1}\}$ of $(n-1)$ positive integers such that $p_i + r_j$ is a prime for all $(i,j) \in \{1,2,...,m\} \times \{1,2,...,n \}$. 
\end{prop}
 
\begin{proof}
Let $X$ and $Y$ be the two partite sets of $K_{m,n}$. Index the vertices of $X = \{x_1, x_2, ..., x_m\}$ and $Y = \{y_1, y_2, ..., y_n\}$ in increasing order. Let $p_i = \frac{x_i+y_1}{2}$ and $r_j = \frac{y_{j+1}-y_1}{2}$ for $(i,j) \in \{1,2,...,m\} \times \{1,2,...,n-1 \}$. Note that $p_i + r_j = \frac{x_i+y_{j+1}}{2}$ is a prime for each $(i,j) \in \{1,2,...,m\} \times \{1,2,...,n-1 \}$.
\end{proof}  

\noindent
Now we will prove some conditions for a complete bipartite subgraph of $ \mathcal{G}_\infty $ with the aid of the following two lemmas. Let us denote $\Nat\cup\set{0}$ by $\Nat_0$.

\begin{lem}\label{lem n6}
Let $a,b \notin 6\Nat_0$ and  $ab \in E(\mathcal{G}_\infty)$, then $ \mid a - b \mid = 6$.
\end{lem}    

\begin{proof}
Let $a > b$. Then $a = p + q$ and $b = p - q$ for some odd primes $p$ and $q$. If $p,q \neq 3$, then $p$ and $q$ are each of the form $6k+1$ or $6k -1$. Then either $6\mid(p+q)=a$ or $6\mid(p-q)=b$ contradicting the assumption of the lemma. Hence $q=3$ (as $p \geq q \geq 3)$ which implies $a - b = 2q = 6$. 
\end{proof}

\begin{lem}\label{lem d6}
Let $a,b \in 6\Nat_0$ and  $ab \in E(\mathcal{G}_\infty)$ with $a \geq b$, then $a=6$  and $b=0$. 
\end{lem}   
    
\begin{proof}
As both  $a$ and $b$ are divisible by 6, both $\frac{a+b}{2}$ and $\frac{a-b}{2}$ are divisible by 3. Since they are primes, $\frac{a+b}{2} = 3 = \frac{a-b}{2}$. Therefore, $a=6, b=0$.
\end{proof} 

\begin{thm}\label{kmn}
Let $K_{m,n}$ be a subgraph of $\mathcal{G}_\infty $ with partite sets $X$ and $Y$  such that $m,n > 2$. Then either $X \subset 6 \Nat_0$ and $Y \cap 6 \Nat_0 = \emptyset$ or $Y \subset 6\Nat_0$ and $X \cap 6\Nat_0 = \emptyset$.
\end{thm}

\begin{proof}
If neither $X \not\subset 6 \Nat_0$ nor $X \cap 6\Nat_0 \ne \emptyset$, then  either there exist $a \in X \cap 6\Nat_0$ and $\left\{ b,c \right\} \subset X \cap (\Nat_0\setminus 6\Nat_0)$ or there exist $\left\{ b,c \right\} \subset X \cap 6\Nat_0$ and $a \in X \cap (\Nat_0\setminus 6\Nat_0)$. 

\vspace{1em}
\noindent
If $a \in X \cap 6\Nat_0$ and $\{ b,c \} \subset X \cap (\Nat_0\setminus 6\Nat_0)$, then $\mid Y \cap 6 \Nat_0 \mid \leq 1$ by Lemma~\ref{lem d6}. Again for any $d \in Y \cap (\Nat_0\setminus 6\Nat_0)$, Lemma~\ref{lem n6} forces  $b-6=d=c+6$ assuming $b > c$, without loss of generality. Hence $ \mid Y \cap (\Nat_0\setminus 6\Nat_0) \mid \leq 1$. Therefore,  $ \mid Y \mid = \mid Y \cap 6 \Nat_0 \mid + \mid Y \cap (\Nat_0\setminus 6\Nat_0) \mid \leq 2$, a contradiction.

\vspace{1em}
\noindent
If $\{ b,c \} \subset X \cap 6\Nat_0$ and $a \in X \cap (\Nat_0\setminus 6\Nat_0)$, then $Y \cap 6 \Nat_0 = \emptyset$ as otherwise each vertex of $Y \cap 6 \Nat_0$ must be adjacent to both $b$ and $c$ forcing them to be the same vertex by  Lemma~\ref{lem d6}. On the other hand,  if $d \in Y \cap (\Nat_0\setminus 6\Nat_0)$, then $d = a-6$ or $d = a+6$ by  Lemma~\ref{lem n6}. This implies $ \mid Y \mid = \mid Y \cap 6 \Nat_0 \mid + \mid Y \cap (\Nat_0\setminus 6\Nat_0) \mid \leq 2$, a contradiction. So either $X \subset 6\Nat_0$ or $X \cap 6\Nat_0 = \emptyset$. If $X \subset 6\Nat_0$, then $Y \cap 6\Nat_0 = \emptyset$ by Lemma~\ref{lem d6}. If $X \cap 6\Nat_0 = \emptyset$, then $Y \cap (\Nat_0\setminus 6\Nat_0) = \emptyset$ by Lemma~\ref{lem n6}. For if $x\in X$ is adjacent to $x-6,x+6\in Y$, then $x+6$ can be adjacent to only $x$ and $x+12$. But $x+12$ is not adjacent to $x-6$.
\end{proof}
    
\noindent
In the next result we will also capture the case where at least one of the partite sets have exactly two vertices while the other one has at least four of them.     

\begin{thm}
Let $K_{2,n}$ be a subgraph of $\mathcal{G}_\infty $ with partite sets $X$ and $Y$  such that $|X| = 2$ and $|Y| = n > 3$. Then either $X \subset 6 \Nat_0$ and $Y \cap 6 \Nat_0 = \emptyset$, or $X \cap 6\Nat_0 = \emptyset$ and $\mid Y \cap (\Nat_0\setminus 6\Nat_0) \mid \leq 1$.
\end{thm}
    
\begin{proof} Suppose $X = \{ a,b \}$. Without loss of generality, let $a \in 6 \Nat_0$ and $b \in (\Nat_0\setminus 6\Nat_0)$. 
Then $\mid Y \cap 6 \Nat_0 \mid \leq 1$ and  $\mid Y \cap (\Nat_0\setminus 6\Nat_0) \mid \leq 2$ by Lemma~\ref{lem d6} and \ref{lem n6}. Therefore, $\mid Y \mid \leq 3$, a contradiction. Hence either $X \subset 6\Nat_0$ or $X \cap 6\Nat_0 = \emptyset$. If $X \subset 6\Nat_0$, then $Y \cap 6\Nat_0 = \emptyset$ by Lemma~\ref{lem d6}. If $X \cap 6\Nat_0 = \emptyset$, then $a$ and $b$ can have at most one common neighbour $c$ such that $c \in (\Nat_0\setminus 6\Nat_0)$ by  Lemma~\ref{lem n6}. 
\end{proof}

\noindent
Now let us try to understand the structure of independent sets in $\mathcal{G}_\infty$. Of course, as $\mathcal{G}_\infty$
is a bipartite graph, there are at least two distinct (and disjoint) independent sets in the form of the two partite sets. 
But how big can an independent set consisting of only consecutive (non-negative) even numbers be? We answer this question in the following result.

\begin{thm}
There exist arbitrarily large independent sets containing consecutive (non-negative) even numbers in $ \mathcal{G}_\infty$. 
\end{thm}

\begin{proof} 
Given any $n\in\Nat$, the  set $R= \{(2n)! + 2, (2n)! + 3, ... ,(2n)! + (2n) \}$ is a set of  consecutive composite numbers. 
Hence no two vertices in the set
$$S = \{ (2n)! + 2, (2n)! + 4, ..., (2n)! + 2n \}$$ 
are adjacent to each other as  $\frac{a+b}{2} \in R$ for all $a,b \in S$. Hence $S$ is an independent set containing $n$ consecutive (non-negative) even numbers in $ \mathcal{G}_\infty$. 
\end{proof} 

\section{Conclusions}

We conclude the paper with an interesting observation that the graphs $\mathcal{G}_{\mathcal{E}_n^*}(\mathcal{P}_1)$ are Hamiltonian for all even $n$ with $4\leqslant n\leqslant 58$, where $\mathcal{E}_n^*=\mathcal{E}_n \smallsetminus \{ 0\}$ and $\mathcal{P}_1=\mathcal{P}\cup \{ 1\}$ (cf. Appendix A). Since the graph $\mathcal{G}_{\mathcal{E}_n^*}(\mathcal{P}_1)$ is bipartite, there cannot be any odd cycle in the graph. But it follows from the above observations that $\mathcal{G}_{\mathcal{E}_n^*}(\mathcal{P}_1)$ has a Hamiltonian path (i.e., a spanning path) for all odd $n$ with $5\leqslant n\leqslant 57$ for if $\mathcal{G}_{\mathcal{E}_{2m}^*}(\mathcal{P}_1)$ is Hamiltonian, then deleting the vertex corresponding to $2m$ from any of its Hamiltonian cycle, we get a Hamiltonian path of $\mathcal{G}_{\mathcal{E}_{2m-1}^*}(\mathcal{P}_1)$. Thus $\mathcal{G}_{\mathcal{E}_n^*}(\mathcal{P}_1)$ has a Hamiltonian path for all $n$ with $4\leqslant n\leqslant 58$. The following is an interesting Hamiltonian path of $\mathcal{G}_{\mathcal{E}_{58}^*}(\mathcal{P}_1)$ that starts with $2$, ends at $116$ and covers all even integers in between them:

\medskip

\noindent $2, 4, 6, 8, 14, 12, 10, 16, 18, 20, 26, 32, 30, 28, 34, 24, 22, 36, 38, 44, 42, 40, 46, 48, 58, 60, 62, 56, 50, 72, 70,$

\noindent $64, 54, 52, 66, 68, 74, 84, 82, 76, 90, 88, 78, 80, 86, 92, 102, 100, 94, 108, 98, 96, 106, 112, 114, 104, 110, 116.$

\vspace{1em}
\noindent
Let us call two even natural numbers {\em conjugate} to each other if $\set{a,b}=\set{p-q,p+q}$ for some $p,q$ each of which is either an odd prime or $1$. We see that there is a sequence of even natural numbers up to $1000$ such that any two consecutive numbers in this sequence are conjugate to each other (cf. Appendix B). Now these observations lead to the following questions:

\begin{enumerate}
\item Does there exist a sequence of all even natural numbers such that any two consecutive numbers in this sequence are conjugate to each other?
\item If the answer to the above question is negative, then what is the least value of $m$ such that $\mathcal{G}_{\mathcal{E}_{2m}^*}(\mathcal{P}_1)$ is not Hamiltonian?
\end{enumerate}

\vspace{1em}\noindent
\textbf{Acknowledgements:}

\vspace{.2cm}\noindent 
{\small We are grateful to the learned referees for their valuable suggestions and corrections which definitely improved and enhanced the paper. This research is partially supported by Jadavpur University under RUSA 2.0 Scheme (R-11/742/19 dated 27.06.2019) to the third author and partially by NSF under award CCF-1907738, which supported the second author in part.}

\bibliographystyle{amsplain}

\newpage
\begin{center}
{\large\bf Appendix A}
\end{center}

{\footnotesize
$$\begin{array}{|c|l|}
\hline
\text{Number of} & \text{Hamiltonian Cycle} \\
\text{Vertices} & \\
\hline
4 & (4,2,8,6,4) \\
6 & (4,6,8,2,12,10,4) \\
8 & (4,2,8,14,12,10,16,6,4) \\
10 & (4,2,8,6,16,10,12,14,20,18,4) \\
12 & (4,2,8,6,16,10,12,22,24,14,20,18,4) \\
14 & (4,2,8,6,28,18,16,22,12,26,20,14,24,10,4) \\
16 & (4,2,8,6,16,10,12,22,24,14,20,26,32,30,28,18,4) \\
18 & (4,2,8,6,16,10,12,14,20,18,28,34,24,22,36,26,32,30,4) \\
20 & (4,2,8,6,16,10,12,22,36,38,24,14,20,26,32,30,28,34,40,18,4) \\
22 & (4,2,8,6,16,10,12,14,20,18,28,34,40,42,32,26,36,22,24,38,44,30,4) \\
24 & (4,2,8,6,16,10,12,14,20,18,28,30,32,26,36,22,24,34,40,46,48,38,44,42,4)\\
26 & (4,2,8,6,16,10,12,14,20,18,28,34,40,46,36,22,24,50,44,38,48,26,32,30,52,42,4) \\
28 & (4,2,8,6,16,10,12,14,20,18,28,30,56,50,44,38,24,22,36,46,40,34,48,26,32,54,52,42,4) \\
30 & (4,2,8,6,16,10,12,14,20,18,28,30,56,50,44,38,24,22,36,46,40,34,48,58,60,26,32,54,52,42,4) \\
32 & (4,2,8,6,16,10,12,14,20,18,28,30,32,26,36,22,24,34,40,46,48,38,44,50,56,62,60,58,64,54,52, \\
 & \null\hfill 42,4) \\
34 & (4,2,8,6,16,10,12,14,20,18,28,30,32,26,36,22,24,34,40,46,48,38,44,50,56,62,60,58,64,54,68, \\
 & \null\hfill 66,52,42,4) \\
36 & (4,2,8,6,16,10,12,14,20,18,28,30,32,26,36,22,24,34,40,46,48,38,44,50,72,70,64,58,60,62,56, \\
 & \null\hfill 66,68,54,52,42,4) \\
38 & (4,2,8,6,16,10,12,14,20,18,28,30,32,26,36,22,24,34,40,46,48,38,44,50,56,62,60,58,64,54,52, \\
 & \null\hfill 66,68,74,72,70,76,42,4) \\
40 & (4,2,8,6,16,10,12,14,20,18,28,30,32,26,36,22,24,34,40,42,52,54,68,74,48,38,44,50,56,62,72, \\
 & \null\hfill 46,60,58,64,70,76,66,80,78,4) \\
42 & (4,2,8,6,16,10,12,14,20,18,28,30,32,26,36,22,24,34,40,42,44,38,48,46,60,58,64,70,76,82,84, \\
 & \null\hfill 50,56,62,72,74,68,54,52,66,80,78,4) \\
44 & (4,2,8,6,16,10,12,14,20,18,28,30,32,26,36,22,24,34,40,42,52,66,68,74,48,38,44,50,56,62,84, \\
 & \null\hfill 58,64,70,76,82,60,46,72,86,80,54,88,78,4) \\
46 & (4,2,8,6,16,10,12,14,20,18,28,30,32,26,36,22,24,34,40,42,44,38,48,46,60,58,64,70,76,82,84, \\
 & \null\hfill 50,56,62,72,74,68,54,52,66,92,86,80,78,88,90,4) \\
48 & (4,2,8,6,16,10,12,14,20,18,28,30,32,26,36,22,24,34,40,42,44,38,48,46,60,58,64,70,76,82,96, \\
 & \null\hfill 50,56,62,72,94,84,74,68,54,52,66,92,86,80,78,88,90,4) \\
50 & (4,2,8,6,16,10,12,14,20,18,28,30,32,26,36,22,24,34,40,46,48,38,44,42,100,94,72,62,56,50,96, \\
 & \null\hfill 98,60,58,64,70,76,82,84,74,68,54,52,66,92,86,80,78,88,90,4) \\
52 & (4,2,8,6,16,10,12,14,20,18,28,30,32,26,36,22,24,34,40,46,48,38,44,42,104,102,100,94,72,62,56, \\
 & \null\hfill 50,96,98,60,58,64,70,76,82,84,74,68,54,52,66,92,86,80,78,88,90,4) \\
54 & (4,2,8,6,16,10,12,14,20,18,28,30,32,26,36,22,24,34,40,46,48,38,44,42,104,102,100,106,108,94,72, \\
 & \null\hfill 62,56,50,96,98,60,58,64,70,76,82,84,74,68,54,52,66,92,86,80,78,88,90,4) \\
56 & (4,10,16,6,8,14,14,20,18,28,30,32,26,36,22,24,34,40,46,48,38,44,42,100,102,104,110,108,106,112, \\
 & \null\hfill 90,88,78,80,86,92,66,52,54,68,74,84,82,76,70,64,58,60,98,96,50,56,62,72,94,12,2,4) \\
58 & (6,4,2,8,14,12,10,16,18,20,26,32,30,28,34,24,22,36,38,44,42,40,46,48,58,60,62,56,50,72,70,64,54, \\
 &  \null\hfill 52,66,68,74,84,82,76,90,88,78,80,86,92,102,116,110,104,114,112,106,96,98,108,94,100,6) \\

\hline
\end{array}$$
}

\newpage
\begin{center}
{\large\bf Appendix B}

\vspace{2em}
A Hamiltonian path in the graph $\mathcal{G}_{\mathcal{E}_{500}^*}(\mathcal{P}_1)$,\\ 
(A sequence of even natural numbers up to $1000$ where \\
any pair of consecutive numbers are conjugate to each other).
\end{center}

\vspace{2em}

$\{22, 16, 10, 4, 2, 8, 14, 20, 26, 32, 54, 52, 42, 80, 86, 92, 102, 112, 106, 100, 94, 12, 514, 408,\\
 394, 400, 138, 284, 278, 288, 250, 204, 658, 664, 30, 136, 142, 120, 74, 68, 6, 340, 334, 328, 294,\\
 152, 146, 528, 586, 372, 362, 356, 350, 344, 18, 28, 34, 40, 46, 36, 430, 436, 442, 576, 662, 656,\\
 650, 644, 638, 24, 338, 636, 710, 672, 682, 516, 502, 496, 798, 844, 642, 776, 770, 324, 218, 228,\\
 926, 932, 390, 508, 330, 292, 210, 304, 310, 316, 222, 244, 238, 264, 482, 476, 282, 416, 422, 720,\\
 598, 604, 610, 616, 622, 504, 82, 76, 70, 64, 58, 420, 578, 300, 322, 996, 998, 984, 982, 900, 914,\\
 908, 906, 800, 794, 180, 122, 84, 842, 836, 78, 88, 114, 412, 406, 108, 166, 160, 154, 148, 366, 592,\\
 450, 764, 758, 60, 778, 768, 634, 628, 66, 812, 806, 72, 626, 468, 206, 540, 698, 704, 582, 940, 934,\\
 588, 674, 680, 686, 552, 646, 640, 762, 752, 630, 808, 814, 708, 746, 740, 738, 700, 594, 872, 882,\\
 892, 870, 788, 786, 172, 834, 688, 694, 948, 986, 980, 974, 968, 750, 268, 274, 280, 286, 252, 374,\\
 492, 466, 460, 454, 312, 614, 620, 618, 676, 670, 972, 970, 964, 942, 716, 722, 696, 818, 824, 830,\\
 828, 946, 960, 242, 236, 230, 224, 702, 920, 954, 248, 254, 260, 266, 272, 606, 308, 314, 320, 306,\\
 388, 414, 332, 666, 472, 726, 728, 654, 632, 246, 712, 706, 732, 554, 560, 566, 572, 966, 928, 714,\\
 724, 730, 736, 742, 660, 262, 924, 898, 816, 958, 936, 902, 864, 118, 780, 862, 856, 850, 804, 962,\\
 876, 718, 684, 938, 944, 950, 956, 498, 220, 226, 232, 126, 352, 346, 48, 38, 44, 50, 56, 62, 96, 98,\\
 104, 110, 116, 198, 784, 790, 648, 734, 612, 190, 196, 202, 192, 890, 884, 878, 840, 886, 880, 874,\\
 852, 866, 860, 854, 912, 326, 888, 434, 428, 894, 832, 826, 820, 774, 548, 930, 212, 150, 296, 290,\\
 132, 410, 456, 562, 480, 494, 488, 510, 532, 426, 452, 446, 440, 318, 524, 518, 360, 478, 216, 418,\\
 424, 342, 404, 398, 276, 470, 464, 458, 624, 298, 600, 766, 772, 270, 448, 186, 952, 486, 652, 474,\\
 608, 534, 692, 546, 580, 574, 568, 354, 512, 506, 500, 378, 368, 174, 208, 214, 168, 158, 156, 602,\\
 596, 590, 584, 258, 200, 194, 188, 234, 392, 386, 380, 162, 164, 170, 176, 182, 144, 490, 484, 522,\\
 140, 134, 128, 90, 256, 462, 556, 558, 536, 918, 520, 526, 432, 550, 544, 538, 396, 922, 916, 438,\\
 364, 370, 376, 382, 240, 302, 444, 530, 384, 542, 336, 178, 184, 402, 124, 130, 348, 358, 756, 838,\\
 564, 782, 792, 802, 744, 910, 904, 822, 896, 810, 848, 858, 868, 570, 796, 690, 748, 754, 760, 678,\\
 668, 846, 992, 990, 976, 978, 988, 994, 1000\}$

\end{document}